\documentclass[11pt]{amsart}
\usepackage{graphicx}
\usepackage{amsthm}
\usepackage{amsmath}
\usepackage{amssymb}
\usepackage{amsbsy}
\usepackage{xcolor}
\usepackage{graphics}
\usepackage{color}
\numberwithin{equation}{section}
\usepackage{enumerate}
\usepackage{url}
\usepackage{mathtools}

\usepackage{amsfonts}
\usepackage{graphicx}
\usepackage{epstopdf}
\usepackage{algorithmic}
\DeclareGraphicsExtensions{.eps,.pdf,.png,.jpg}

\newtheorem{theorem}{Theorem}
\theoremstyle{remark}
\newtheorem{remark}{Remark}
\theoremstyle{definition}
\newtheorem{algorithm}{Algorithm}

\DeclareMathOperator{\diag}{diag}
\DeclareMathOperator{\proj}{proj}
\newcommand{\norm}[1]{\left\Vert#1\right\Vert}

\newcommand{\bigo}{{\mathcal O}}
\newcommand{\bigop}[1]{{\mathcal O}\left({#1}\right)}

\newcommand{\f}{\frac}

\newcommand{\R}{\mathbb{R}}

\newcommand{\anorm}{\norm{\,\cdot \,}}

\newcommand{\eps}{\varepsilon}

\newcommand{\nr}[1]{\ensuremath{\left\|{#1}\right\|}}
\newcommand{\nrs}[1]{\ensuremath{\|{#1}\|}} 

\usepackage{hyperref}
\hypersetup{
  pdftitle={Nigam-Pollock article},
  pdfauthor={N. Nigam and S. Pollock}
  pdfsubject={Mathematics},    
  pdfkeywords={iterative methods, numerical analysis}, 
  colorlinks=true,        
  linkcolor=red,          
  citecolor=blue,         
  filecolor=magenta,      
  urlcolor=cyan           
}
\begin{document}
\title[Eigenvalue extrapolation]{A simple extrapolation method for clustered eigenvalues}
\author[N. Nigam]{Nilima Nigam}
\email{nigam@math.sfu.ca}

\author[S. Pollock]{Sara Pollock}
\email{s.pollock@ufl.edu}

\address{Department of Mathematics, Simon Fraser University, Burnaby, BC, CA}
\address{Department of Mathematics, University of Florida, Gainesville, FL, US} 

\thanks{SP is supported in part by the National Science Foundation NSF-DMS 1852876.
         NN is supported through the Discovery Grants program of the Natural Sciences and          Engineering Research Council of Canada.}
\date{\today}
\begin{abstract}
This paper introduces a simple variant of the power method.
It is shown analytically and numerically to accelerate convergence to the dominant 
eigenvalue/eigenvector pair; and, it is particularly effective for problems featuring
a small spectral gap.
The introduced method is a one-step extrapolation technique that uses a linear combination
of current and previous update steps to form a better approximation of the
dominant eigenvector. The provided analysis shows the method converges 
exponentially with respect to the ratio between the two largest eigenvalues,
which is also approximated during the process.  An augmented technique is 
also introduced, and is shown to stabilize the early stages of the iteration.
Numerical examples are provided to illustrate the theory and demonstrate the methods.
\end{abstract}
\maketitle

\section{Introduction}\label{sec:intro}
The power method is a standard tool for capturing the dominant eigenvalue /eigenvector
pair of a matrix.  Its advantages include simplicity and ease of implementation.
Since it relies only on repeated matrix-vector products, the method can be run without
explicit formation of the underlying system matrix.  Applications for recovering 
dominant eigenpairs include stability analysis of PDE systems \cite{BaOs91}, 
and principal component analysis (PCA), where often only the dominant
eigencomponents are of interest \cite{DSHMRX19}.  However, the convergence of the 
method is proportional to the rate $r = |\lambda_2/\lambda_1|$, or $r^2$ in the
Hermitian case \cite[Chapter 8]{GoVa13}, where the eigenvalues are ordered by
descending magnitude.

The purpose of this paper is to introduce an accelerated version of the power method
based on a one-step extrapolation. This method is shown to demonstrate asymptotically 
exponential convergence with respect to the ratio $r$. 
It will be shown to improve convergence to the dominant eigenpair for 
positive semi-definite systems, where the ratio $r$ is close to unity.  
More generally, it will be shown in the numerical results to be effective on 
{\em dominantly} positive definite 
systems, where the eigenvalues of magnitude close to $\lambda_1$ are positive, but
smaller eigenvalues may be of either sign.

Each eigenvector iterate $u_{k+1}$ is
formed by a linear combination of $v_{k+1}$ and $v_k$, where 
$v_{k+1} = A u_{k+1}/\nr{u_{k+1}}$.
The most basic form of the iteration is 
$u_{k+1} = (1-\gamma_{k})v_{k+1} + \gamma_k v_k$, where 
$\gamma_k = -\nr{d_{k}}/\nr{d_{k-1}}$, 
the residual $d_k$ is given by
$u_k - \lambda_{k-1} x_{k-1}^\gamma$, $\lambda_{k-1}$ 
is the Rayleigh quotient, and $x_{k-1}^\gamma$ is the $A$-preimage of the accelerated
$u_k$.  The method requires the storage of two additional vectors, and does not 
require significant additional computations other than a Rayleigh quotient and residual 
norm at each iteration. The method is motivated by the idea that if the initial iterate
$u_0$ is a linear combination of the first two eigenvectors, then the method 
converges exponentially by choice of an extrapolation parameter that approximates
$-r^k$ at iteration $k$.  This motivation will be given in detail in section 
\ref{sec:ideal},
and an analysis of the extrapolation method will be given as the main result
in section \ref{sec:simple}.

Extrapolation methods to accelerate the power iteration have been of recurring 
interest over the past several decades.
The well-known Aitken's acceleration is discussed in the context
of the power iteration in \cite[Chapter 9]{wilkinson65}, and more recently in 
\cite[Chapter 5]{kamvar10}, the latter as applied to Markov matrices for 
PageRank computations. More recent extrapolation techniques in \cite{BWM20,DSHMRX19}
reduce the algorithmic complexity sometimes associated with implementing 
Aitken's process (which should not necessarily be applied at every step), 
and succeed in defining stable accelerations which improve the 
convergence rate of the power iteration. 

The approach introduced in \cite{BWM20} defines a parametrized 
power method based on viewing the power iteration as a nonstationary Richardson 
or gradient descent method.
It is defined by 
$y_{k+1} = (1-\omega_k \mu_k)y_k/\nr{y_k} + \omega_k Ay_k/\nr{y_k}$, where $\mu_k$
is a Rayleigh quotient and $\omega_k$ is a parameter. The improvement in the convergence
for a given problem using this method
is dependent on appropriate choice of the parameter $\omega_k$.
Another extrapolation approach is presented in \cite{DSHMRX19} in the context of 
a stochastic iteration for PCA. 
That method, first introduced in a deterministic
setting, takes the form $u_{k+1} = Au_k - \beta u_{k-1}$. The product $\beta u_{k-1}$ is 
called a momentum term, with $\beta$ the momentum parameter, due to its motivation 
by the heavy ball method of \cite{polyak64}.
The ideal value of $\beta$ is given as
$\lambda_2^2/4$, and given this information, this method is shown
to accelerate convergence of the power iteration for semi-definite problems, and
particularly those featuring a small spectral gap.
The method is sensitive to the value of $\beta$, however, and as
$\lambda_2$ is not generally available, it must be approximated  
to effectively run the method.
A heuristic approach that performs an initial sequence of iterations to generate 
such an approximation is suggested in \cite{DSHMRX19},
but it is not guaranteed to produce a sufficiently good approximation of $\lambda_2$.
The proposed method of the present paper differs, as it generates a 
convergent sequence of approximations to $\lambda_1$ and $r = \lambda_2/\lambda_1$, 
as the method progresses. Its implementation does not depend on the knowledge of any
{\em a priori} unknown quantities.

This form of acceleration may not necessarily compete with sophisticated preconditioned
subspace iteration methods such as \cite{DSYG18,GoYe02,knyazev01}, 
particularly in finding more than one extremal eigenpair of a normal matrix.  
However, the ease of implementation, even with existing codes that
use a power iteration, make extrapolation methods attractive as they can succeed in 
substantially reducing computational time with only a few additional vector operations.
For finding a single dominant eigenpair of a symmetric matrix, 
a numerical comparison between the 
present methods and the inverse free preconditioned Krylov subspace methods
of \cite{GoYe02,QuYe10} is included in subsection \ref{subsec:ifpks}, 
and indicates these methods can have similar runtimes. 
In subsection \ref{subsec:ex1}, the presented extrapolated methods are also 
demonstrated to accelerate the power iteration on nonsymmetric systems.

The remainder of the paper is organized as follows. In subsection \ref{subsec:alg}, 
the introduced 
simple and augmented extrapolation methods are stated. 
In section \ref{sec:ideal}, the simple method, alg. \ref{alg:simple}, 
is motivated by an idealized
extrapolation algorithm that makes use of information such as the value of $r$ that
in general is {\em a priori} unavailable.
The main result, theorem \ref{thm:simple} is presented in section \ref{sec:simple}, 
and shows that
the idealized method of section \ref{sec:ideal} is well-approximated by the practical
alg. \ref{alg:simple}. 
Section \ref{sec:augment} includes a discussion of the augmented method, 
alg. \ref{alg:simp-aug}, and numerical results are given in section 
\ref{sec:numerics}.

\subsection{Extrapolation algorithms}\label{subsec:alg}
Define the inner product for $u, v \in \R^n$, by $(u,v) = u^Tv$, and let
$\anorm$ be the induced $l_2$ norm. 
First, the power method is stated for 
notational convenience.

\begin{algorithm}{Power method}
\label{alg:power}
\begin{algorithmic}
\STATE{Choose $u_0$, set $h_0 = \nr{u_0}$ and $x_0 = h_0^{-1}u_0$.}
\FOR{$k = 0, 1, \ldots$}
\STATE{Set $u_{k+1} = A x_k$, and $h_{k+1} = \nr{u_{k+1}}$}
\STATE{Set $x_{k+1} = h_{k+1}^{-1} u_{k+1}$, $\lambda_k = (u_{k+1}, x_k)$, 
       and $d_{k+1} = {u_{k+1} - \lambda_k x_k}$}
\STATE{{STOP} if $\nr{d_{k+1}} <$ \tt{tol}  }
\ENDFOR
\end{algorithmic}
\end{algorithm}

The first novel method introduced requires two power iterations to start, 
as two consecutive
residuals are required to define each extrapolation parameter $\gamma_k$.  In the
analysis that follows, $m$ iterations of alg. \ref{alg:power} will be used before the
simple extrapolation method below is started.
\begin{algorithm}{Simple extrapolation}
\label{alg:simple}
\begin{algorithmic}
\STATE{Choose $u_0$ and set $h_0 = \nr{u_0}$}
\FOR{$k = 0,1$}
\STATE{Set $x_k = h_k^{-1}u_k$, $u_{k+1} = v_{k+1} = A x_k$, $h_{k+1} = \nr{u_{k+1}}$}
\STATE{Set $\lambda_k = (u_{k+1},x_k)$, and $d_{k+1} = {u_{k+1} - \lambda_k x_k}$}
\ENDFOR

\FOR{$k = 2, 3, \ldots$}
\STATE{$\displaystyle \begin{aligned}
\text{Set }~  &x_{k} = h_k^{-1} u_k \\
  &v_{k+1} = A x_k 
\end{aligned}$}
\STATE{Compute $\gamma_k = -\nr{d_{k}}/\nr{d_{k-1}}$  }
\STATE{$\displaystyle \begin{aligned} 
\text{Set }~ 
  &u_{k+1} = (1-\gamma_k) v_{k+1} + \gamma_{k}v_k,  \quad h_{k+1} = \nr{u_{k+1}} \\
  &x_{k}^\gamma = (1-\gamma_k) x_{k} + \gamma_{k}x_{k-1} \\
  &\lambda_k = (u_{k+1},x_{k}^\gamma)/(x_k^\gamma,x_k^\gamma) \\
  &d_{k+1} = {u_{k+1} - \lambda_k x_k^\gamma} 
\end{aligned}$} 
\STATE{{STOP} if $\nr{d_{k+1}} <$ \tt{tol}  }
\ENDFOR
\end{algorithmic}
\end{algorithm}
The next method, alg. \ref{alg:simp-aug}, 
is designed to start after two (rather than $m$) power iterations,
by augmenting the calculation of the extrapolation parameter with additional 
eigenvalue-specific information. 
Algorithm \ref{alg:simp-aug} is a one-parameter
family of methods where larger values of the parameter $\eta$ reduce the influence of 
the extrapolation.
In the remainder, alg. \ref{alg:simple} will be referred to as the simple method, and
alg. \ref{alg:simp-aug} will be called the augmented method.
\begin{algorithm}{Augmented extrapolation}
\label{alg:simp-aug}
\begin{algorithmic}
\STATE{Choose $u_0$ and parameter $\eta\ge1$,  and set $h_0 = \nr{u_0}$}
\FOR{$k = 0,1$}
\STATE{Set $x_k = h_k^{-1}u_k$, $u_{k+1} = v_{k+1} = A x_k$, $h_{k+1} = \nr{u_{k+1}}$}
\STATE{Set $\lambda_k = (u_{k+1},x_k)$, $d_{k+1} = {u_{k+1} - \lambda_k x_k}$}
\ENDFOR
\STATE{Set $p_1 = (v_2 - u_1,x_1)$}

\FOR{$k = 2, 3, \ldots$}
\STATE{$\displaystyle \begin{aligned}
\text{Set }~  &x_{k} = h_k^{-1} u_k \\
  &v_{k+1} = A x_k \\
  &p_k = (v_{k+1}-u_k,x_k) 
\end{aligned}$}
\STATE{Compute $\gamma_k = -\big(\nr{d_k}^2 + p_k^2\big)^{1/2}/ 
 \big(\nr{d_{k-1}}^2 + (\eta p_{k-1})^2\big)^{1/2}$}
\STATE{$\displaystyle \begin{aligned} 
\text{Set }~ 
   &u_{k+1} = (1-\gamma_k) v_{k+1} + \gamma_{k}v_k,  \quad h_{k+1} = \nr{u_{k+1}} &\\
   &x_{k}^\gamma = (1-\gamma_k) x_{k} + \gamma_{k}x_{k-1} \\
   &\lambda_k = (u_{k+1},x_{k}^\gamma)/(x_k^\gamma,x_k^\gamma) \\
   &d_{k+1} = {u_{k+1} - \lambda_k x_k^\gamma} 
\end{aligned}$} 
\STATE{{STOP} if $\nr{d_{k+1}} <$ \tt{tol}  }
\ENDFOR
\end{algorithmic}
\end{algorithm}

\begin{remark}[On alg. \ref{alg:simp-aug}]\label{rem:eta}
The $\lambda_k$ are all true Rayleigh quotients for the accelerated iterates.  
The scalar quantities $p_k$ are defined by $p_k = (v_{k+1}- u_k, x_k)$. 
They are denoted by $p$ because each is the projection of the difference between
the latest extrapolated iterate and its image under $A$, along the normalized
iterate $x_k$.  Another view is
\[
p_k = (v_{k+1} - u_k,x_k) = (Ax_k - h_k x_k,x_k) = (Ax_k,x_k) - h_k,
\]
the difference between the pre-extrapolation Rayleigh quotient $(Ax_k, x_k)$,
and the norm of the previous accelerated iterate, both approximations of the
dominant eigenvalue.

Some intuition on the parameter $\eta$ in alg. \ref{alg:simp-aug}, can be gained
by considering the limiting cases.  If $\eta$ is sufficiently large, 
then $\gamma_k$ tends toward zero, by which alg. \ref{alg:simp-aug} reduces to 
the alg. \ref{alg:power}.  This helps to explain why alg. \ref{alg:simp-aug} effectively
replaces running some number $m$ iterations of the power iteration before applying
the simple extrapolation, alg. \ref{alg:simple}, as done in the analysis below.  
A moderately chosen parameter $\eta$ reduces the influence of the 
extrapolation at the beginning of the iteration. The influence of $\eta$ then decreases
as $p_k$ is dominated by $\nr{d_k}$, as the first eigencomponent is resolved.
Further detail is given in section \ref{sec:augment}. 
\end{remark}

The analysis of the simple extrapolation method of alg. \ref{alg:simple} which is
presented below in section \ref{sec:simple} is complicated by the $l_2$ normalization 
factors $h_j$ and the inexactness of $\gamma_j$ as an approximation to $-r^j$, where
$r = \mu/\lambda < 1$ is the ratio of the first two eigenvalues, assumed to be positive.
To motivate how the method works without these complications, an idealized analysis
is next presented. The aforementioned complications are removed by normalizing 
by the principal eigenvalue $\lambda$ rather than the $l_2$ norm of iterate $u_j$,
and by supposing $\gamma_j$ is $-r^j$. This is not possible for most practical 
purposes, where $\lambda$ and $r$ are {\em a priori} unknown.
The notable exception is in PageRank algorithms, for which the (dominant) 
positive-definiteness is generally not satisfied, and for which appropriate accelerations
have been well-developed elsewhere, for instance in 
\cite{GoGr06,HKMG03,HWHSG21,KHMG03} and \cite[Chapter 5]{kamvar10}, 
and the references therein.
However, supposing this information were available, a classical extrapolation 
viewpoint illustrates how exponential convergence can be achieved starting from a linear 
combination of the first two eigenvectors. In section \ref{sec:augment} and the 
numerical results of section \ref{sec:numerics}, it will be illustrated that the starting 
vector need not be restricted to the first two eigencomponents.

\section{Motivation: idealized extrapolation}\label{sec:ideal}
Suppose $m$ iterations of the power method have been run on $v_0 = \varphi + c \psi$
with normalization $v_k = \lambda^{-1}Av_{k-1}$, where $\lambda$ and $\mu$ are the 
first two leading eigenvalues with corresponding orthonormal
eigenvectors $\varphi$ and $\psi$.
Let $r = \mu/\lambda$. Then
\begin{align*}
v_m  = \varphi + c r^{m} \psi, \quad
v_{m+1}  = \varphi + c r^{m+1} \psi,
\quad
v_{m+2}  = \varphi + c r^{m+2} \psi.
\end{align*}

Now let's choose $\gamma_{m+1}$ 
to form $u_{m+2} = (1-\gamma_{m+1}) v_{m+2} + \gamma_{m+1} v_{m+1}$ to 
get a higher-order deterioration on the second eigencomponent
\begin{align}\label{eqn:extrap-001}
u_{m+2} &= (1-\gamma_{m+1})v_{m+2} + \gamma_{m+1} v_{m+1}
 = \varphi + ((1-\gamma_{m+1})r + \gamma_{m+1}) cr^{m+1} \psi.
\end{align}
From \eqref{eqn:extrap-001}, the choice $\gamma_{m+1} = -r$ allows
$u_{m+2} = \varphi + c r^{m+3} \psi$,
which improves on $v_{m+2}$ by a factor of $r$, and yields
$v_{m+3} = \varphi + c r^{m+4} \psi$.
Notice that we could have chosen $\gamma_{m+1} = -r/(1-r)$, to eliminate the entire
second eigencomponent. This, however, leads to a highly unstable method, 
particularly if $r$ is close to one, which is the case we are most interested in. 
Repeat the process for $\gamma_{m+2}$
\begin{align*}
u_{m+3} &= (1-\gamma_{m+2})v_{m+3} + \gamma_{m+2} v_{m+2}
 = \varphi + ((1-\gamma_{m+2})r^2 + \gamma_{m+2}) c r^{m+2}\psi. 
\end{align*}
The choice of $\gamma_{m+2}= -r^2$ allows
$u_{m+3} = \varphi + c r^{m+6} \psi$, and
$v_{m+4}  = \varphi + c r^{m+7} \psi$. 

Continuing inductively, suppose
\begin{align}\label{eqn:pextrap007}
v_{m+k}  = \varphi + c (r^{m + 1 + \sum_{l=1}^{k-1}l})\psi, \quad 
v_{m+k+1} = \varphi + c (r^{m + 1 + \sum_{l=1}^{k}l})\psi. 
\end{align}
Then for $\gamma_k = -r^k$ we have
\begin{align}\label{eqn:extrap007}
u_{m+k+1} 
& = (1-\gamma_{m+k}) v_{m+k+1} + \gamma_{m+k} v_{m+k}
\nonumber \\ 
&= \varphi 
+ \left( (1-\gamma_{m+k})r^k + \gamma_{m+k} \right)
(r^{m+1+\sum_{l=1}^{k-1}l})c\psi
\nonumber \\
& = \varphi + c(r^{m + \sum_{l = 1}^{k+1}l}) \psi, 
\nonumber \\
v_{m+k+2} & = \varphi + c(r^{m + 1+ \sum_{l = 1}^{k+1}l}) \psi. 
\end{align}
which completes the induction.  

Next, we demonstrate that this idea works in a practical sense, by 
alg. \ref{alg:simple}, where the extrapolation parameter $\gamma_{m+k}$ is defined by 
$-\nr{d_{m+k}}/\nr{d_{m+k-1}}$.
As shown in the next section and numerically in section \ref{sec:numerics}, 
this approximates $-r^k$ up to $\bigop{r^{2(m-1)+k}}$, which is  
sufficient to obtain the exponential convergence seen above.

\begin{remark}
The choice of $\gamma_{m+k} =-\nr{d_{m+k}}/\nr{d_{m+k-1}}$ is not the only possibility
to give this approximation to $-r^k$.  One could also choose
$\hat \gamma_{m+k} = -\nr{w_{m+k}}/\nr{w_{m+k-1}}$, where $w_j = v_{j+1} - u_j$. 
Numerically, the two methods appear essentially equivalent, but the analysis 
was found to be simpler defining the extrapolation parameter in terms of the residuals.
\end{remark}

\section{Analysis of simple extrapolation}\label{sec:simple}
Suppose the power iteration, alg. \ref{alg:power}, 
is run for some given number of iterations $m$, 
before switching to the  simple extrapolation, alg. \ref{alg:simple}.
As in section \ref{sec:ideal}, we will
suppose the initial iterate $u_0$ is some linear combination  of the first two 
(dominant) normalized eigenvectors: 
$\varphi$ with eigenvalue $\lambda$, and $\psi$ with eigenvalue $\mu$,
with $(\varphi,\psi) = 0$.
Let $r = \mu/\lambda$. 
In the analysis that follows, we will suppose $r > 1/2$.
Supposing $u_0 = \varphi + c \psi$, for some $c \in \R$, after $n \ge 1$ power iterations
\begin{align}\label{eqn:simp-001}
x_{n-1} &= (1 + c^2r^{2(n-1)})^{-1/2}(\varphi + c r^{n-1} \psi) 
         = \delta_{n-1}^{-1} (\varphi + c r^{n-1} \psi),
\nonumber \\
u_n = v_n & = \lambda (1 + c^2r^{2(n-1)})^{-1/2} (\varphi + c r^n \psi)
            = \lambda \delta_{n-1}^{-1} (\varphi + c r^n \psi),
\nonumber \\
\lambda_{n-1} &= (u_{n},x_{n-1}) = \lambda \f{1+c^2 r^{2n-1}}{1+c^2 r^{2(n-1)}}
 = \lambda \delta_{n-1}^{-2}(1+c^2 r^{2n-1}),
\end{align}
where 
$\delta_n \coloneqq (1 + c^2 r^{2n})^{1/2}$. 
The residual $d_n = u_n - \lambda_{n-1} x_{n-1}$ 
is then given by
\begin{align}\label{eqn:simp-002}
d_n &= \lambda \delta_{n-1}^{-1}
\left\{ (\varphi + c r^n \psi) - \f{1 + c^2 r^{2n-1}}{1 + c^2 r^{2(n-1)}}
(\varphi + cr^{n-1} \psi) \right\}
\nonumber \\
 &= \lambda \delta_{n-1}^{-1}\left\{
\varphi \left( 1 - \f{1 + c^2 r^{2n-1}}{1 + c^2 r^{2(n-1)}}  \right) 
+ cr^{n-1}\psi \left(r - \f{1 + c^2 r^{2n-1}}{1 + c^2 r^{2(n-1)}}
 \right) \right\}
\nonumber \\
 &= \lambda(1-r) cr^{n-1} \delta_{n-1}^{-3}\left\{
c r^{n-1} \varphi - \psi \right\}.
\end{align}
Similarly, 
$d_{n+1} = \lambda (1-r) cr^n \delta_n^{-3}\left\{
cr^n \varphi - \psi \right\}.$
Taking the ratio of $\nr{d_{n+1}}$ and $\nr{d_n}$ for $n = m$ 
to perform the first extrapolation, we have
\begin{align}\label{eqn:simp-002a}
\nr{d_m} &= \f{\lambda(1-r) cr^{m-1}}{\delta_{m-1}^2}=
\f{\lambda(1-r) cr^{m-1}}{1 + c^2 r^{2(m-1)}}, 
\nonumber \\
\nr{d_{m+1}} &= \f{\lambda(1-r) cr^{m}}{\delta_m^2} =  
\f{\lambda(1-r) cr^{m}}{1 + c^2 r^{2m}}, 
\end{align}
which then yields
\begin{align}\label{eqn:simp-003}
\f{\nr{d_{m+1}}}{\nr{d_m}}
 = r  \f{\delta_{m-1}^2}{\delta_m^2}
 = r \left( 1 + \f{c^2 r^{2(m-1)}(1-r^2)}{1 + c^2 r^{2m}} \right)
 = r  \left( 1 + \eps^d_{m+1} \right),
\end{align}
where
\begin{align}\label{eqn:simp-004}
0 <  \eps^d_{m+1} =  \f{c^2 r^{2(m-1)}(1-r^2)}{\delta_m^2} <  c^2 r^{2(m-1)}(1-r^2).
\end{align}

The first accelerated iterate 
$u_{m+2} = (1 - \gamma_{m+1})v_{m+2} + \gamma_{m+1} v_{m+1}$, 
is then given by
\begin{align}\label{eqn:simp-005}
&  \lambda \delta_{m+1}^{-1}\left\{
(1-\gamma_{m+1})  ( \varphi + cr^{m+2}\psi)
+ \gamma_{m+1}(\varphi + c r^{m+1} \psi)
\right\}
\nonumber \\
& = \lambda \delta_{m+1}^{-1} \left\{
\varphi \left( 1 - \gamma(1 - \delta_{m+1}\ /\delta_m) \right) 
+ cr^{m+1} \psi 
\left( r - \gamma_{m+1}(r  - \delta_{m+1}/\delta_m) \right)
\right\}.
\end{align}

The expression for $u_{m+2}$ can be simplified using a little algebra, 
and the inequality 
$\sqrt{1-x} = 1-y$ for some $x/2 < y < x$, if $0 < x < 1$, by
\begin{align}\label{eqn:simp-006}
1 - \delta_{m+1}\ /\delta_m  
= 1 - \left(1 - \f{c^2r^{2m}(1-r^2)}{ 1 + c^2r^{2m}} \right)^{1/2} 
= \eps^u_{m+2},
\end{align}
where $0 < \eps^u_{m+2} < c^2 r^{2m}(1-r^2)$.
From \eqref{eqn:simp-006}, it follows that
\begin{align}\label{eqn:simp-007}
r - \delta_{m+1}\ /\delta_m = r - (1 - \eps^u_{m+2}).
\end{align}
Applying \eqref{eqn:simp-006} and \eqref{eqn:simp-007} to \eqref{eqn:simp-005}
with $\gamma_{m+1} = -\nr{d_{m+1}}/\nr{d_m}$ expressed as \eqref{eqn:simp-003},
allows the expansion of $u_{m+2}$ by 
\begin{align}\label{eqn:simp-008}
&\lambda \delta_{m+1}^{-1}\left\{
\varphi \left(1\! - \gamma_{m+1}\eps^u_{m+2} \right)
 + cr^{m+1}\psi \left(r - \gamma_{m+1}(r - 1 + \eps^u_{m+2} \right)
\right\}
\nonumber \\
& \!= \! \lambda \delta_{m+1}^{-1} \! \left\{
\varphi \left(1\! +r \eps^u_{m+2} (1\! +\! \eps^d_{m+1}) \right)
\! +\! cr^{m+1}\psi \left(r +r(1+\eps^d_{m+1})(r - 1 +\! \eps^u_{m+2} \right)
\right\}
\nonumber \\
& \!=\! \lambda \delta_{m+1}^{-1} \! \left\{
\varphi \left(1\! +r \eps^u_{m+2} (1\! + \!\eps^d_{m+1}) \right)
 \!+\! cr^{m+2}\psi \left(-\eps^d_{m+1} \!+\! (1 + \eps^d_{m+1} )(r  +\! \eps^u_{m+2} \right)
\right\} \!\!.
\end{align}
The term multiplying $\psi$ can be written as
\begin{align*}
cr^{m+3}\left(1 + \f{\eps^u_{m+2}}{r}(1 + \eps^d_{m+1})
- \f{1-r}{r} \eps^d_{m+1} \right),
\end{align*}
which allows, for $r > 1/2$, and $m$ large enough so $c^2 r^{2(m-1)}(1-r^2) < 1$,
that the accelerated iterate $u_{m+2}$ may be written as 
\begin{align}\label{eqn:simp-009}
u_{m+2} & = \lambda \delta_{m+1}^{-1}\left\{
\varphi(1 + \theta_{m+2}) + cr^{m+3} (1 + \eta_{m+2}) \psi\right\},
\nonumber \\
 0 < \theta_{m+2} &< 2 r \eps^u_{m+2} < 2 c^2 r^{2m+1}(1-r^2),
\nonumber \\
|\eta_{m+2}| &< 2 \eps^d_{m+1} < 2 c^2 r^{2(m-1)}(1-r^2). 
\end{align}

Continuing the iteration up through the computation of $u_{m+3}$ before moving 
onto the inductive step, 
$x^\gamma_{m+1}$ the (or a, if $A$ is singular) $A$-preimage of $u_{m+2}$,
$x_{m+2}$ the $l_2$ normalization of $u_{m+2}$ from \eqref{eqn:simp-009}, and
$v_{m+3}$ are given by
given by
\begin{align}\label{eqn:simp-010}
x^\gamma_{m+1} & = \delta_{m+1}^{-1}\left\{
\varphi(1 + \theta_{m+2}) + cr^{m+2} (1 + \eta_{m+2}) \psi \right\},
\nonumber \\
x_{m+2} & = \delta_{m+2}^{-1}\left\{
\varphi(1 + \theta_{m+2}) + cr^{m+3} (1 + \eta_{m+2}) \psi \right\},
\nonumber \\
v_{m+3} & = \lambda \delta_{m+2}^{-1}\left\{
\varphi(1 + \theta_{m+2}) + cr^{m+4} (1 + \eta_{m+2}) \psi \right\}, ~\text{ with }
\nonumber \\
\quad \delta_{m+2} &= 
\left( (1 + \theta_{m+2})^2 + c^2 r^{2(m+3)} (1 + \eta_{m+2})^2  \right)^{1/2}.
\end{align}

The next Rayleigh quotient 
$\lambda_{m+1} = (u_{m+2},x^\gamma_{m+1})/\nr{x^\gamma_{m+1}}^2$, is 
given by
\begin{align}\label{eqn:simp-011}
\lambda_{m+1} & = \lambda \f{(1 + \theta_{m+2})^2 + c^2 r^{2(m+2)+1}(1 + \eta_{m+2}^2)}
{(1 + \theta_{m+2})^2 + c^2 r^{2(m+2)}(1 + \eta_{m+2}^2)}. 
\end{align}
To expedite the process of computing the residual, first notice that
\begin{align}\label{eqn:simp-012}
&1 -\f{(1 + \theta_{m+2})^2 + c^2 r^{2(m+2)+1}(1 + \eta_{m+2}^2)}
{(1 + \theta_{m+2})^2 + c^2 r^{2(m+2)}(1 + \eta_{m+2})^2}
\nonumber \\ 
& = \f{(1-r)c^2 r^{2(m+2)}(1 + \eta_{m+2}^2)}
{(1 + \theta_{m+2})^2 + c^2 r^{2(m+2)}(1 + \eta_{m+2})^2}, ~\text{ and }
\nonumber \\
&r\!-\!\f{(1 + \theta_{m+2})^2 + c^2 r^{2(m+2)}(1 + \eta_{m+2}^2)}
{(1 + \theta_{m+2})^2 + c^2 r^{2(m+2)}(1 + \eta_{m+2}^2)}
\!=\! \f{(r-1)(1 + \theta_{m+2})^2}
{(1 + \theta_{m+2})^2 + c^2 r^{2(m+2)}(1 + \eta_{m+2}^2)}.
\end{align}

Applying \eqref{eqn:simp-012} along with the expansions 
\eqref{eqn:simp-009}-\eqref{eqn:simp-011},
the residual $d_{m+2} = u_{m+2} - \lambda_{m+1}x^\gamma_{m+1}$ is then
\begin{align*}
d_{m+2} & =
\f{\lambda}{\delta_{m+1}} 
\f{cr^{2(m+2)}(1-r)(1 + \theta_{m+2})(1 + \eta_{m+2})}
{(1 + \theta_{m+2})^2 + c^2 r^{2(m+2)}(1 + \eta_{m+2}^2)}
\nonumber \\
\times &
\left\{
cr^{m+2}(1 + \eta_{m+2}) \varphi + (1 + \theta_{m+2})\psi
\right\},
\end{align*}
and therefore its norm is
\begin{align}\label{eqn:simp-013}
\nr{d_{m+2}} & = \f{\lambda cr^{2(m+2)}(1-r)(1 + \theta_{m+2})(1 + \eta_{m+2})}
{\delta_{m+1}\delta_{m+2}}.
\end{align}
Now, from \eqref{eqn:simp-004} and \eqref{eqn:simp-013} the next ratio of residual
norms is given by 
\begin{align}\label{eqn:simp-014}
\f{\nr{d_{m+2}}}{\nr{d_{m+1}}}
&= r^2 \f{(1 + \theta_{m+2})(1 + \eta_{m+2}) \delta_m^2}{\delta_{m+1} \delta_{m+2}}
\nonumber \\
&= r^2 \f{(1 + \eta_{m+2}) (1 + c^2 r^{2m})}
{(1 + c^2 r^{2(m+1)})^{1/2} (1 + c^2 r^{2(m+3)}(1+\eta_{m+2})^2/(1+\theta_{m+2})^2)^{1/2}}
\nonumber \\
&= r^2 \f{(1 + \eta_{m+2}) (1 + c^2 r^{2m})}{1 + \alpha c^2 r^{2(m+1)}},
\end{align}
where $\alpha$ lies in the interval between $1$ and 
$r^4(1+\eta_{m+2})^2/(1+\theta_{m+2})^2.$ 
This last simplification is justified by the observation 
that both terms in the denominator of \eqref{eqn:simp-014} are the square roots of
positive perturbations of one. Then 
\begin{align}\label{eqn:simp-014a}
\f{\nr{d_{m+2}}}{\nr{d_{m+1}}}
& = r^2 \left( 1 + \f{\eta_{m+2}(1 +  c^2 r^{2m}) + c^2 r^{2m}(1 - r^2 \alpha)}
{1 + \alpha c^2 r^{2(m+1})} \right)
 = r^2 (1 + \eps^d_{m+2}),
\nonumber\\ 
|\eps^d_{m+2}|&< 2|\eta_{m+2}| + c^2 r^{2m}(1-r^2 \alpha)
< c^2 r^{2(m-1)}( 2(1-r^2) +r^{2}(1-r^2 \alpha)).
\end{align}

From the expressions for $v_{m+2}$ and $v_{m+3}$ in \eqref{eqn:simp-001}
and \eqref{eqn:simp-009}, 
respectively, the next extrapolated iterate 
$u_{m+3} = (1-\gamma_{m+2})v_{m+3} + \gamma_{m+2}$, can
be written as
\begin{align}\label{eqn:simp-0in}
u_{m+3} & = \lambda \delta_{m+2}^{-1} \Bigg\{
\varphi \left( (1 + \theta_{m+2}) 
- \gamma_{m+2}\left((1 + \theta_{m+2}) - \f{\delta_{m+2}}{\delta_{m+1}}\right) \right)
\nonumber \\
& +  c r^{m+2}\psi \left( r^2(1 + \eta_{m+2})
- \gamma_{m+2} \left( r^2(1 + \eta_{m+2}) - \f{\delta_{m+2}}{\delta_{m+1}}\right)
\right)\Bigg\}.
\end{align}
Applying \eqref{eqn:simp-014},
the terms multiplying the principal eigenvector $\varphi$ in \eqref{eqn:simp-0in}
reduce to
\begin{align}\label{eqn:simp-014d}
&1 + \theta_{m+2}(1 - \gamma_{m+2})- \gamma_{m+2}
\left(1 - \f{\delta_{m+2}}{\delta_{m+1}} \right)
\nonumber \\
& = 
1 + \theta_{m+2}(1 - \gamma_{m+2})-
r^2(1+\theta_{m+2)}(1+\eta_{m+2})\delta_m^2\left( 
\f{1}{\delta_{m+1}\delta_{m+2}} - \f{1}{\delta_{m+1}^2}
\right)
\nonumber \\
& = 1 + \theta_{m+3}, \quad \theta_{m+3} = 
\bigo\left( r^{2m+1}\right).
\end{align}
Hereafter, we are concerned with tracking the perturbations in terms of 
powers of $r$. The lowest-order term in \eqref{eqn:simp-014d} comes from 
$\theta_{m+2}$, with a higher-order term of $\bigo\left(r^{2(m+2)} \right),$
resulting from $r^2(\delta_{m+1}^{-1}\delta_{m+2}^{-2}-\delta_{m+1}^{-2})$, by
a similar calculation to that in  \eqref{eqn:simp-014}.
Together, \eqref{eqn:simp-0in} and \eqref{eqn:simp-014d} show the extrapolation does 
not do much damage to preserving the 
component of $u_{m+3}$ along $\varphi$. Next, consider the terms multiplying
the second eigenvector $\psi$, which the extrapolation is designed to reduce
from $\bigo\left({r^{m+4}}\right)$ (as in $v_{m+3}$) to $\bigo(r^{m+6})$.
Applying the expression for $\gamma_{m+2} = -\nr{d_{m+2}}/\nr{d_{m+1}}$, from the first
line of \eqref{eqn:simp-014}, allows 
\begin{align}\label{eqn:simp-014d1}
&r^2(1+\eta_{m+2}) + \gamma_{m+2}\f{\delta_{m+2}}{\delta_{m+1}}
\nonumber \\
& = r^2(1+\eta_{m+2})\left( 1 - (1-\theta_{m+2})
\left( 1 + \f{c^2 r^{2m}(1-r^2)}{1+c^2 r^{2m}} \right) \right)
\nonumber \\
& = r^4\cdot\bigop{r^{2(m-1)}},
\end{align}
and from \eqref{eqn:simp-014a} the remaining term satisfies
$\gamma_{m+2}r^2(1 +\eta_{m+2}) = r^4 
\bigop{r^{2(m-1)}}$.
Putting these two estimates into \eqref{eqn:simp-0in} yields
\begin{align}\label{eqn:simp-014e}
&c r^{m+2} \left( r^2(1 + \eta_{m+2})
- \gamma_{m+2} \left( r^2(1 + \eta_{m+2}) - \f{\delta_{m+2}}{\delta_{m+1}}\right) \right)
\nonumber \\
& = cr^{m+6}\cdot 
\bigop{r^{2(m-1)}}.
\end{align}
Summarizing, the extrapolated iterate $u_{m+3}$ can be expressed as
\begin{align}\label{eqn:simp-014f}
u_{m+3} & = \lambda \delta_{m+2}^{-1} \left\{(1 + \theta_{m+3}) \varphi 
+ cr^{m+6}(1 + \eta_{m+3}) \psi\right\},
\nonumber \\
\theta_{m+3} & = \bigop{r^{2m+1}},
\quad \eta_{m+3}  = \bigop{r^{2(m-1)}}.
\end{align}

The general inductive step is next shown in the following theorem.
\begin{theorem}\label{thm:simple}
Let $A$ be a positive semi-definite matrix with leading
orthonormal eigenvectors $\varphi$ and $\psi$ 
corresponding to eigenvalues $\lambda$ and $\mu$, with $\lambda > \mu$ 
and $r = \mu/\lambda > 1/2$.
Let $u_0 = \varphi + c \psi$, and suppose $m$ iterations of the power method,
alg. \ref{alg:power} are run before switching to the simple extrapolation method,
alg. \ref{alg:simple}, where $m$ is large enough so that $c^2 r^{2(m-1)}(1-r^2) < 1$.  
Then for $j \ge 1$,  
the extrapolation parameter $\gamma_{m+j} = -\nr{d_{m+j}}/\nr{d_{m+j-1}}$ 
and the accelerated iterate $u_{m+j}$ satisfy
\begin{align}\label{eqn:thm-simple}
\gamma_{m+j} &= -r^j(1 + \eps^d_{m+j}), \quad \eps^d_{m+j} = \bigop{r^{2(m-1)}}, 
\nonumber \\
u_{m+j+1} &= \lambda \left\{\varphi(1 + \hat\theta_{m+j+1}) + r^{m+\sum_{l = 1}^{j+1}l} 
(1 + \hat\eta_{m+j+1})\psi \right\}, ~\text{ with }
\nonumber \\
\hat \theta_{m+j+1} & = \bigop{r^{2m+1}},
~\text{ and }~
\hat \eta_{m+j+1}  = \bigop{r^{2(m-1)}}.
\end{align}
\end{theorem}
\begin{proof}
The base steps with $j=1$ and $j=2$ are established in the preceding discussion.
There, and in the following, it will be shown that 
\[ u_{m+j+1} = \lambda \delta_{m+j}^{-1}
\left\{\varphi(1 + \theta_{m+j+1}) + r^{m+\sum_{l = 1}^{j+1}l} 
(1 + \eta_{m+j+1})\psi \right\},\] 
where for $j \ge 2$, 
$\delta_{m+j} = 1 + \bigop{r^{2m+1}}$, which is sufficient to establish
\eqref{eqn:thm-simple}.
Now, let us proceed inductively under the following hypotheses, where
the indices $J = m+j$,  $s = \sum_{l = 1}^{j-1}l$, and $t = s + j$,
are introduced for notational brevity.
Suppose, in accordance with \eqref{eqn:simp-009} and \eqref{eqn:simp-010}, that
\begin{align}\label{eqn:simp-015h}
\gamma_{m+k}& = -r^k(1 + \eps^d_{m+k}), \quad \eps^d_{m+k} = \bigop{r^{2(m-1)}},
\nonumber \\
u_{m+k+1} &= \lambda \delta_{m+k}^{-1}\left\{
\varphi(1 + \theta_{m+k+1}) + r^{m\sum_{l = 1}^k l} (1 + \eta_{m+k+1}) \psi
\right\},
\nonumber \\
\theta_{m+k+1} & = \bigop{r^{2m+1}}, \quad \eta_{m+k+1}  = \bigop{r^{2(m-1)}},
\end{align}
for $k < j$, where the normalization factors are
\begin{align}\label{eqn:simp-015a}
\delta_{m+k} & = \left( (1 + \theta_{m+k})^2 + c^2 r^{2(m+\sum_{l = 1}^k l)}
(1 + \eta_{m_k})^2
\right)^{1/2}. 
\end{align}

We will proceed to compute 
$\gamma_{J} = -\nr{d_{J}}/\nr{d_{J-1}}$, and the extrapolated iterate $u_{J+1}$. 
From \eqref{eqn:simp-015h} and alg. \ref{alg:simple}, we have
\begin{align}\label{eqn:simp-015}
u_{J} & = \lambda \delta_{J-1}^{-1}
\left\{ (1 + \theta_{J})\varphi + cr^{m+t}(1 + \eta_{J}) \psi \right\}, 
\nonumber \\
x_{J-1}^\gamma & = \lambda \delta_{J-1}^{-1}
\left\{ (1 + \theta_{J})\varphi + cr^{m+t-1}(1 + \eta_{J}) \psi \right\}. 
\end{align}
The Rayleigh quotient $\lambda_{J-1}$ is then given by
\begin{align}\label{eqn:simp-016}
\lambda_{J-1}  = \f{(u_{J}, x^\gamma_{J-1})}{(x^\gamma_{J-1}, x^\gamma_{J-1})}
= \lambda \f{(1 + \theta_{J})^2 + c^2 r^{2(m+t)-1}(1 + \eta_{J})^2}
{(1 + \theta_{J})^2 + c^2 r^{2(m+t-1)}(1 + \eta_{J})^2}.
\end{align}
The residual $d_{J}$ is given from \eqref{eqn:simp-015} and \eqref{eqn:simp-016} by
\begin{align}\label{eqn:simp-017}
u_{J} - \lambda_{J-1}x^\gamma_{J-1}
 = \f{\lambda}{\delta_{J-1}}\left\{
(1 + \theta_{J})(1-Z) + cr^{m+t-1}(1+\eta_{m+z})(r-Z)\psi
\right\},
\end{align}
where $Z$ is the ratio that multiplies $\lambda$ in the Rayleigh quotient 
\eqref{eqn:simp-016}. The factors $1-Z$ and $r-Z$ are then
\begin{align}\label{eqn:simp-018}
(1 - Z) & = \f{c^2 r^{2(m+t-1)}(1 + \eta_{J})^2(1-r) }
{(1 + \theta_{J})^2 + c^2 r^{2(m+t-1)}(1 + \eta_{J})^2}, 
\nonumber \\
(r-Z) & = \f{(1 + \theta_{J})^2(r-1)}
{(1 + \theta_{J})^2 + c^2 r^{2(m+t-1)}(1 + \eta_{J})^2}. 
\end{align}
Applying \eqref{eqn:simp-018} to \eqref{eqn:simp-017} yields the residual
$d_{J}$ as
\begin{align}\label{eqn:simp-019}
\f{\lambda(1 + \theta_{J})(1 + \eta_{J})(1-r)cr^{m+t-1}}
{\delta_{J-1}\left((1 + \theta_{J})^2 + c^2 r^{2(m+t-1)}(1 + \eta_{J})^2\right)}
\!\left\{
cr^{m+t-1}(1+\eta_{m+z})\varphi - (1 + \theta_{J}) \psi
\right\},
\end{align}
with its norm given by
\begin{align}\label{eqn:simp-019n}
\nr{d_{J}}\!
& =\! \f{\lambda(1 + \theta_{J})(1 + \eta_{J})(1-r)cr^{m+t-1}}
{\delta_{J-1}\left((1 + \theta_{J})^2 + c^2 r^{2(m+t-1)}(1 + \eta_{J})^2
\right)^{1/2}}.
\end{align}
It is useful to note the simplification,
$(1 + \eta_{m+k})/(1+ \theta_{m+k}) = 1 + \eps_k$, where
$\eps_k = \bigop{r^{2(m-1)}}$, $k = {j-1},j$.
Then applying the same process for $\nr{d_{J-1}}$, 
as \eqref{eqn:simp-019}-\eqref{eqn:simp-019n}, the ratio of residuals 
$\nr{d_{J}}/ \nr{d_{J-1}}$ 
is given by
\begin{align}\label{eqn:simp-020}
&  
r^{t-s} \f{(1 + \theta_{J})(1 + \eta_{J})}
         {(1 + \theta_{J-1}) (1 + \eta_{J-1})}
\cdot\f{ \delta_{J-2} \left((1 + \theta_{J-1})^2  + c^2 r^{2(m+s-1)} (1 + \eta_{J-1})^2
\right)^{1/2}}
{
\delta_{J-1} \left(
(1 + \theta_{J})^2 + c^2 r^{2(m+t-1)} (1 + \eta_{J})^2
\right)^{1/2}}
\nonumber \\
&=  
r^{j} \f{(1 + \eta_{J})}
         {(1 + \eta_{J-1})}
\cdot
\f{ \delta_{J-2} \left( 1  + c^2 r^{2(m+s-1)} (1 + \eps_{j-1})^2
\right)^{1/2}}
{
\delta_{J-1} \left( (1 + c^2 r^{2(m+t-1)} (1 + \eps_{j})^2
\right)^{1/2}}
\nonumber \\
& = r^j(1 + \eps^d_{J}), \quad \eps^d_{J} = \bigop{r^{2(m-1)}},
\end{align}
where the lowest-order terms arise from $\eta_{m+k}$, $k = \{j-1,j\}$.
This shows the first equation of \eqref{eqn:thm-simple}.

The next extrapolated iterate is given by 
$u_{J+1} = (1 - \gamma_{J})v_{J+1} + \gamma_{J} v_{J}$, 
where from the inductive hypothesis and alg. \ref{alg:simp-aug}
\begin{align*}
v_{J}  &= \lambda \delta_{J-1}^{-1}
\left\{ (1 + \theta_{J-1})\varphi + cr^{m+s+1}(1 + \eta_{J-1}) \psi \right\}, 
\\
v_{J+1}  &= \lambda \delta_{J}^{-1}
\left\{ (1 + \theta_{J})\varphi + cr^{m+t+1}(1 + \eta_{J}) \psi \right\}. 
\end{align*}
Then rearranging terms, $u_{J+1}$ can be written as
\begin{align}\label{eqn:simp-021}
u_{J+1} & = \lambda \delta_{J}^{-1}\Bigg\{ 
\varphi\left( (1 + \theta_{J}) - \gamma_{J}
\left( (1+ \theta_{J}) - (1+\theta_{J-1})
\f{\delta_{J}}{\delta_{J-1}} \right)\right)
\nonumber \\
& +\psi cr^{m+s+1}\left(r^j(1+\eta_{J}) - \gamma_{J}\left(
r^j(1+\eta_{J}) - (1+\eta_{J-1})\f{\delta_{J}}{\delta_{J-1}}
\right)\right).
\end{align}
Applying $\gamma_{J} = - \nr{d_{J}}/\nr{d_{J-1}}$, given by \eqref{eqn:simp-020},
the coefficient multiplying  $\varphi$ can be written as $(1 + \theta_{J+1})$, 
with perturbation
$\theta_{J+1} = \bigop{r^{2m+1}}$, where the lowest order term is inherited from 
$\theta_{J}$.

Looking more carefully at the terms multiplying $\psi$, we first want to see
the sum of the first and third terms multiplying $\psi c r^{m+s+1}$ 
in \eqref{eqn:simp-021},
is of order at least $r^{2(J-1)}$. Then after factoring $r^{2j}$ 
out of the entire expression, the remaining terms will be of order $r^{2(m-1)}$.
Similarly to the base case in \eqref{eqn:simp-014d1}, the expression for 
$\gamma_{J}$ from the next-to-last line of \eqref{eqn:simp-020} will be used to 
cancel like factors of $\eta_{J-1}$.
\begin{align}\label{eqn:simp-022}
&r^j(1+\eta_{J}) + \gamma_{J}(1+\eta_{J-1})\f{\delta_{J}}{\delta_{J-1}}
\nonumber \\ &
= r^j(1+\eta_{J})\left(1 - \f{\delta_{J-2} \delta_{J}}{\delta_{J-1}^2}
\cdot \left(
\f{ 1+c^2 r^{2(m+s-1)}(1 + \eps_{j-1})^2  }
{   1+c^2 r^{2(m+t-1)}(1 + \eps_{j})^2    }
\right)^{1/2}  \right).
\end{align}
The square-rooted term of \eqref{eqn:simp-022} 
is easily seen to reduce to a term of the form $1 + \bigop{r^{2(m+s-1)}}$.
Defining $s_0 = \sum_{l = 1}^{j-2}$, 
the term contributing the lowest-order perturbation 
$\delta_{J-2}\delta_j/\delta_{J-1}^2$,
can be understood by the factorization
\begin{align}\label{eqn:simp-022a}
 \f{ \left( \left( 1+c^2 r^{2(m+s_0)}(1 + \eps_{j-2})^2 \right)
        \left( 1+c^2 r^{2(m+t)}(1 + \eps_{j})^2 \right) \right)^{1/2}}
      {\left( 1+c^2 r^{2(m+s)}(1 + \eps_{j-1})^2 \right)}
\f{(1 + \theta_{J-2})(1+\theta_{J})}{(1+\theta_{J-1})^2},
\end{align}
where for the $j=3$ case, $\theta_{J-2} = \theta_{m+1} = 0$.
The first ratio of \eqref{eqn:simp-022a} produces a term of the form 
$1 + \bigop{r^{2(m+s_0)}}$, where $2 s_0 - j \ge -1 $ for $j \ge 3$,
so the perturbation is of order at least $r^j \cdot \bigop{r^{2m -1} }$.
The remaining term also produces a perturbation of at least order $r^{2m-1}$
for $j > 3$, and of order $r^{2(m-1)}$ for $j=3$.
For $j=3$, this is because $\theta_{m+3}$ and $\theta_{m+2}$ are both
$1 + \bigop{r^{2m+1}} = 1 + r^3 \cdot \bigop{r^{2(m-1)}}$.
For $j > 3$ we have
\begin{align}\label{eqn:simp-022b}
\f{(1 + \theta_{J-2})(1+\theta_{J})}{(1+\theta_{J-1})^2}
& = 1 + \f{(\theta_{J} - \theta_{J-1}) - (\theta_{J-1} - \theta_{J-2})
+ (\theta_{J} \theta_{J-2} - \theta_{J-1}^2)}
{(1 + \theta_{J-1})^2},
\end{align}
the lowest-order term of which is $(\theta_{J-1} - \theta_{J-2})$. Both
terms in the difference are of order $r^{2m+1}$, and may be analyzed as follows.
For $k \ge 3$, the term multiplying $\varphi$ in the iterate $u_{m+k}$ is given 
({\em cf.} \eqref{eqn:simp-021}) by
\begin{align*}
\lambda \delta_{m+k-1}^{-1}
\left( (1 + \theta_{m+k-1}) \!-\! \gamma_{m+k-1}
\left( (1+ \theta_{m+k-1}) \!-\! (1+\theta_{m+k-2})
\f{\delta_{m+k-1}}{\delta_{m+k-2}} \right)\right),
\end{align*}
by which $\theta_{m+k} = \theta_{m+k-1} - \gamma_{m+k-1}(\theta_{m+k-1} - \theta_{m+k-2})
+ h.o.t.$, where the higher-order terms ($h.o.t.$), will not be consequential.
Together with the inductive hypotheses on $\gamma_k$ and $\theta_k$, 
this shows that $\theta_{J-1} - \theta_{J-2}$ is of at least order ${r^{2m +j-1}}$. 
Applying this
back into \eqref{eqn:simp-022b}, \eqref{eqn:simp-022a} and \eqref{eqn:simp-022} shows
\begin{align*}
r^j(1+\eta_{J}) + \gamma_{J}(1+\eta_{J-1})\f{\delta_{J}}{\delta_{J-1}}
= r^{2j} \cdot \bigop{r^\nu}, ~\text{ with } \nu \ge 2(m-1).
\end{align*}

Finally, from \eqref{eqn:simp-020}, and the inductive hypothesis on $\eta_{m+j}$,
the remaining term of \eqref{eqn:simp-021} that
multiplies $\psi cr^{m+s+1}$, satisfies
\[
-\gamma_{J}r^j(1 + \eta_{J}) = r^{2j}(1 + \eps^d_{J})(1 + \eta_{J}) 
= r^{2j}\left(1 + \bigop{r^{2(m-1)}}\right).
\] 

Putting everything together into \eqref{eqn:simp-021}, and noting that
$m+s + 2j + 1 = m+t + (j+1)$, we have
\begin{align}\label{eqn:simp-024}
u_{J+1} &= \lambda \delta_{J}^{-1}\left\{
\varphi(1 + \theta_{J+1}) + r^{m+t+(j+1)} (1 + \eta_{J+1})
\right\},
\nonumber \\
\theta_{J+1} & = \bigop{r^{2m+1}}, \quad \eta_{J+1}  = \bigop{r^{2(m-1)}},
\end{align}
which establishes \eqref{eqn:thm-simple}.
\end{proof}
\section{Augmenting the simple method}\label{sec:augment}
In this section we discuss the motivation behind the augmented method, 
alg. \ref{alg:simp-aug}.
The key, and only substantial, difference between the augmented method 
and the simple method of alg. \ref{alg:simple}, is the use of the projection 
$p_k:=(v_{k+1}-u_k,x_k) = (Ax_k,x_k)-h_k$ to compute the extrapolation 
parameter
\[
\gamma_k = -\big(\nr{d_k}^2 + p_k^2\big)^{1/2}/ 
\big(\nr{d_{k-1}}^2 + (\eta p_{k-1})^2\big)^{1/2}.
\]
The parameter $\gamma_k$ for the augmented method features a user-defined 
tuning parameter $\eta \ge 1$. 
Large values of $\eta$ reduce the effect of the extrapolation; and, as discussed below,
more moderate values can help resolve transient modes earlier in the iteration.
For general initial iterates, 
the quantity $\gamma_k$ generally decreases as the algorithm converges, 
though the behavior need not be monotone. 
With $\eta$ chosen well, the augmented method often provides faster 
convergence than the simple method to the correct eigenvector.
As demonstrated in subsections \ref{subsec:ex1} and \ref{subsec:ex2}, 
if $\eta$ is chosen too large, the iteration 
remains stable but takes longer to converge, as it then follows
more closely the dynamics of the power iteration.

To better understand the distinction between the simple and augmented
algorithms, we may examine the 
difference between one step of each.
If the simple method is run without any preliminary power iterations 
beyond the first two, then both algorithms have the same 
$u_1$ and $u_2$, and it makes sense to compare outcomes for $u_3$.
 
In practice, the augmented method does not require a full orthogonal basis of 
eigenvectors.  But, for clarity of presentation, suppose $A$ is an $n\times n$ 
matrix, with an orthogonormal basis of eigenvectors $\{\phi_i\}_{i = 1}^n$, 
and spectrum
$$\underbrace{\mu_1,\mu_1,..,\mu_1}_{J},\mu_{J+1},...\mu_M,...\mu_n$$ where 
$\mu_1>\mu_{J+1}\geq ...\mu_{n}\geq 0$, and let an initial iterate $u_0$ be
\[
u_0 = \sum_{i=1}^J c_i \phi_i + c_M \phi_M + \sum_{i>J,i\not=M}^n c_i \phi_i.
\] 
The index $M$ is associated with the largest magnitude component, $|c_M|$, in the 
initial vector. 
 
We will evaluate $u_3$ starting from this initial iterate $u_0$ 
for both alg. \ref{alg:simp-aug} and alg. \ref{alg:simple}. 
The first two iterates $u_{k+1}$, $~k = 0,1$, are then
\begin{align}\label{eqn:aug-001}
u_{k+1} =\frac{1}{h_k}Au_k
= \frac{1}{\prod_{j = 0}^k h_j}
\left[\mu_1^{k+1} \sum_{i=1}^J  c_i \phi_i + \mu_M^{k+1} c_M \phi_M 
+ \sum_{i>J,i\not=M}^n \mu_i^{k+1} c_i \phi_i \right],
\end{align}
with Rayleigh quotients
$\lambda_k = {h_k^{-2}}(Au_k,u_k)$ given by
\begin{align*}
\frac{1}{(\prod_{j=0}^k h_j)^2}
\left[ \mu_1^{2k+1} \sum_{i=1}^J c_i^2 \phi_i + \mu_M^{2k+1} c_M^2 \phi_M 
+  \sum_{i>J,i\not=M}^n  \mu_i^{2k+1} c_i^2 \phi_i \right].
\end{align*}

The residual vectors $d_{k+1} = u_{k+1} - (\lambda_k/h_k) u_k$ are then given by
\[
\frac{1}{h_k} \left[ \mu_1^k \sum_{i=1}^J(\mu_i-\lambda_1) c_i \phi_i 
+ \mu_M^k(\mu_M-\lambda_1) c_M \phi_M + \sum_{i>J,i\not=M}^n \mu_i^k(\mu_i-\lambda_1) 
c_i \phi_i \right]. 
\]
The projection $p_1$ is $p_1= (Ax_1,x_1)-h_1 =\lambda_1-h_1$, and 
the projection $p_2$ is 
\[
p_2=\left(\frac{1}{h_2}Au_2 - u_2, \frac{1}{h_2}u_2\right)
= \frac{1}{h_2^2}(Au_2,u_2) -h_2 = \lambda_2 - h_2.
\] 
For the first extrapolated step of each method, keeping
track of the two different methods $u^{acc}$, where $acc = \{s,au\}$ for the simple
and augmented methods respectively, we have
\begin{align*}
\gamma^{au}_2 &= 
-\left[\|d_2 \|^2+|p_2|^2\right]^{1/2}\slash \left[\|d_1 \|^2+\eta^2|p_1|^2\right]^{1/2},
\quad 
\gamma^{s}_2 = -\left[\|d_2 \|\right]\slash \left[\|d_1 \|\right]
\\
u^{acc}_3 &= \frac{1}{h_2}(1-\gamma_2^{acc})Au_2 
+ \frac{\gamma_2^{acc}}{h_1} Au_1.
\end{align*} 

In particular, from \eqref{eqn:aug-001}, the components of $u_3^{acc}$ which lie in the 
dominant eigenspace for each method are
\begin{align}\label{eqn:aug-002}
\proj_{span(\phi_1,..\phi_J)}u_3^{acc} 
= \mu_1^2\frac{\gamma^{acc}_2}{h_0h_1h_2}\left[ \sum_{i=1}^J c_i (h_2-\mu_1)\phi_i\right].
\end{align}
The ratio of norm of these two projections is then nothing but the ratio of 
the extrapolation parameters
$  {\nrs{\proj_{span(\phi_1,..\phi_J)}u_3^{au}}}/
{\nrs{\proj_{span(\phi_1,..\phi_J)}u_3^{s}}} 
= {\gamma_2^{au}}/{\gamma_2^s}.$
When this ratio is larger than one, $u_3^{au}$ yields a better 
approximation to $\mu_1$ in its Rayleigh quotient. 
This is satisfied when
\[ \left| \frac{\gamma_2^{au}}{\gamma_2^s}\right|^2 
= \frac{\|d_2\|^2+|p_2|^2}{\|d_2\|^2} 
\left(\frac{\|d_1\|^2}{\|d_1\|^2+\eta^2 |p_1|^2}\right) >1 
\qquad \mbox{iff} \qquad 
\left| \frac{p_2}{p_1}\right| >\eta\frac{\|d_2\|}{\|d_1\|}.
\]

Similarly to \eqref{eqn:aug-002}, the projection onto the eigenspace with the 
largest initial coefficient is given by
\begin{align}
\proj_{span(\phi_M)}u_3^{acc} 
= \left( \f{\mu_M}{\mu_1}\right)^2\frac{\gamma^{acc}_2}{h_0h_1h_2} c_M (h_2-\mu_M)
\phi_M.
\end{align}
If $|c_M|>>|c_i|$, as
\[
|c_M|\left(\frac{\mu_M}{\mu_1}\right)^k <<|c_i|,
\] 
the augmented method magnifies the difference between the growth in the dominant
and principle subdominant component.  As the iteration continues and 
subdominant modes are sequentially suppressed, the iteration reduces essentially
to simple method.  Its behavior is then described well by alg. \ref{alg:simple}.
\section{Numerical results}\label{sec:numerics}
In this section we present numerical results that illustrate the theory and
demonstrate the presented methods.  The first two examples 
illustrate convergence rates predicted by the theory, and robustness with respect to 
initial iterates and nonsymmetric perturbations.
The third example compares the simple and augmented methods with
the inverse free preconditioned Krylov subspace methods of \cite{GoYe02,QuYe10}.
The final two examples demonstrate use of the algorithms for finite element 
discretizations of Neumann and Steklov eigenvalue problems. 
\subsection{Example 1: Benchmarking}\label{subsec:ex1}
We will start by looking at a simple problem to verify and illustrate the theory.  
Then we will look at three benchmark examples using matrices of different sizes.

First, consider the simple method alg. \ref{alg:simple}, started after $m=10$ power
iterations, applied to the diagonal matrix $A = \diag([1, 0.9, 0.5, \ldots, 0.5])$.
In the following results, $A$ is $50 \times 50$, but the number of padding entries
of $0.5$ appears inconsequential. The iteration is started with $u_0$ 
a vector of ones. 
Table \ref{tab:ex0} shows $\gamma_{j+1}/\gamma_j$ which according
to theorem \ref{thm:simple} should be approximately $r = 0.9$; $u_j(2)/r^{1 + \ldots +j}$, 
the component of the approximate eigenvector in the second eigendirection, normalized
by $r^{\sum_{l = 1}^j l}$, which should be approximately constant; and, the 
residual $\nr{d_j}$. Each of the quantities behaves as predicted, with the second 
eigencomponent decaying a little faster as the algorithm converges. This simple 
example confirms the theoretical convergence of the eigenvector, along with the 
residual.
\begin{table}
\begin{center}
\begin{tabular}{c|ccccccc}
$j$ & 1 & 2 & 3 & 4 & 5 & 6 & 7  \\
\hline
$\gamma_j/\gamma_{j-1}$ &
- & 0.912 & 0.899 & 0.887 & 0.886 & 0.893 & 0.899  
\\
$u_j(2)/r^{1 +\ldots + j}$ &
0.304 & 0.309 & 0.310 & 0.308 & 0.305 & 0.302 & 0.299
\\
$\nr{d_j}$ & 
2.4e-02 & 1.8e-02 & 1.2e-02 & 7.0e-03 & 3.7e-03 & 1.8e-03 & 7.5e-04
\\ \hline 
j & 8 & 9 & 10 & 11 & 12 & 13 & 14 \\ \hline
$\gamma_j/\gamma_{j-1}$ &
 0.900 & 0.899 & 0.898 & 0.900 & 0.905 & 0.909 & 0.899
\\
$u_j(2)/r^{1 +\ldots + j}$ &
0.295 & 0.291 & 0.287 & 0.285 & 0.281 & 0.266 & 0.200
\\
$\nr{d_j}$ & 
2.9e-04 & 9.9e-05 & 3.0e-05 & 8.5e-06 & 2.2e-06 & 5.0e-07 & 9.9e-08
\end{tabular}
\end{center}
\caption{The ratio of consecutive extrapolation parameters, the component along the
second exact eigenvector scaled by $r^{1 + \ldots _ j}$, 
and the norm of the residual for a diagonal matrix with
$r = 0.9$.}
\label{tab:ex0}
\end{table}

Next we demonstrate the simple extrapolation method of alg. \ref{alg:simple} and
the augmented method of alg. \ref{alg:simp-aug} compared to the power method, 
alg. \ref{alg:power}.
In each of these tests, alg. \ref{alg:simple} is started after $m=40$ initial power 
iterations, and the augmented method is run with parameter $\eta = 40$.
\begin{figure}
\centering
\includegraphics[trim = 5pt 5pt 10pt 10pt,clip = true, width=0.3\textwidth]
{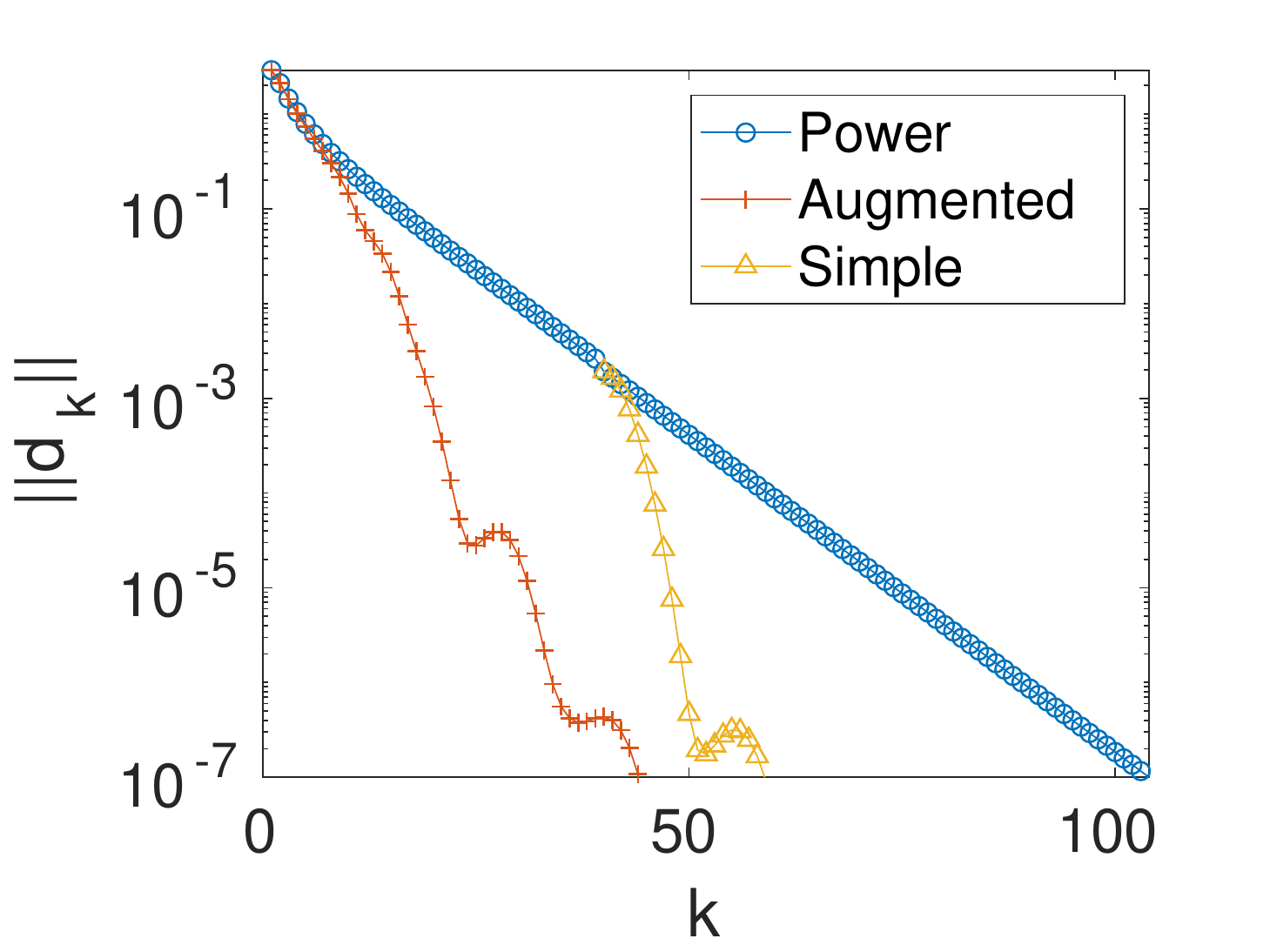}
\includegraphics[trim = 5pt 5pt 10pt 10pt,clip = true, width=0.3\textwidth]
{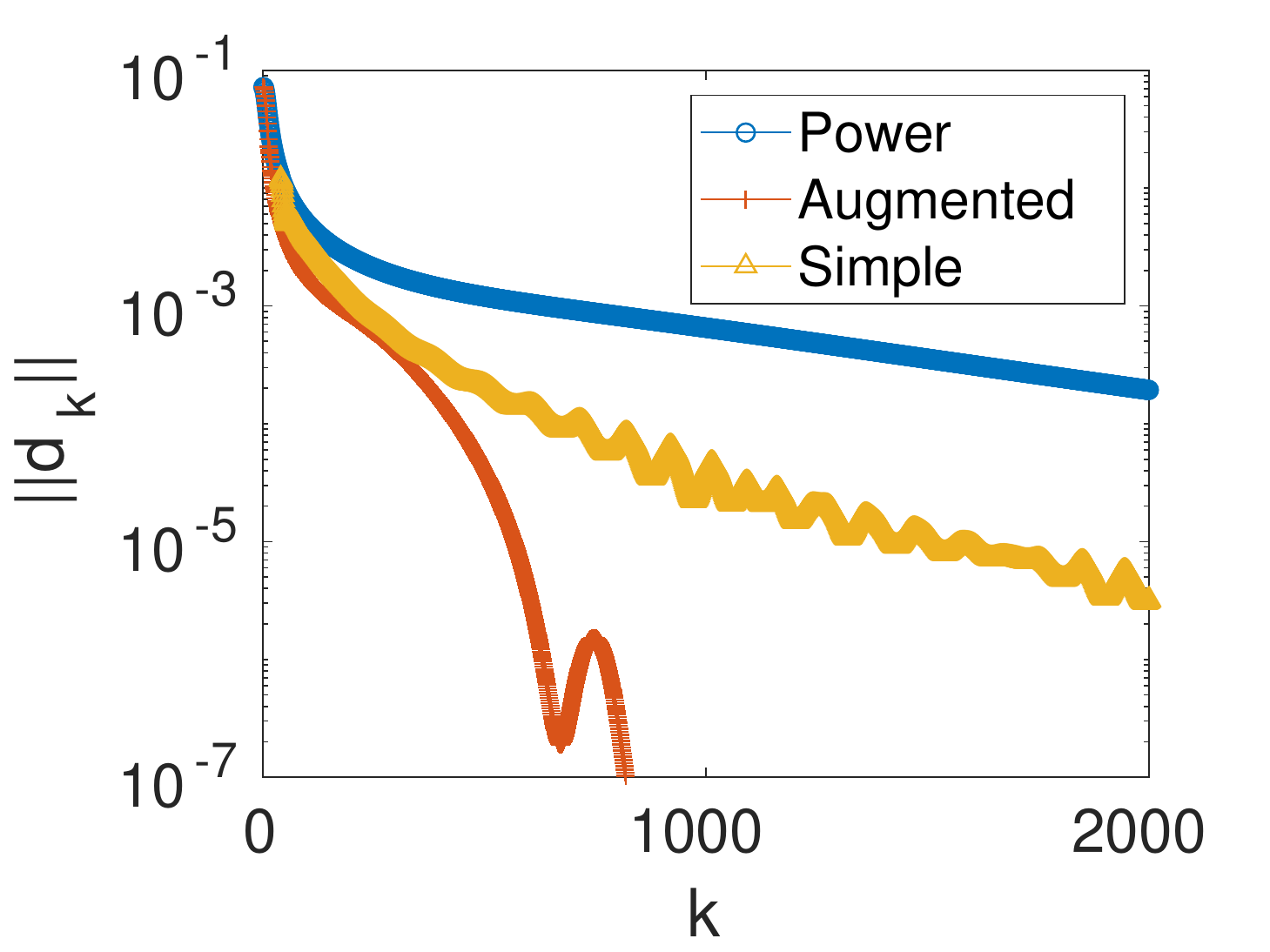}
\includegraphics[trim = 5pt 5pt 10pt 10pt,clip = true, width=0.3\textwidth]
{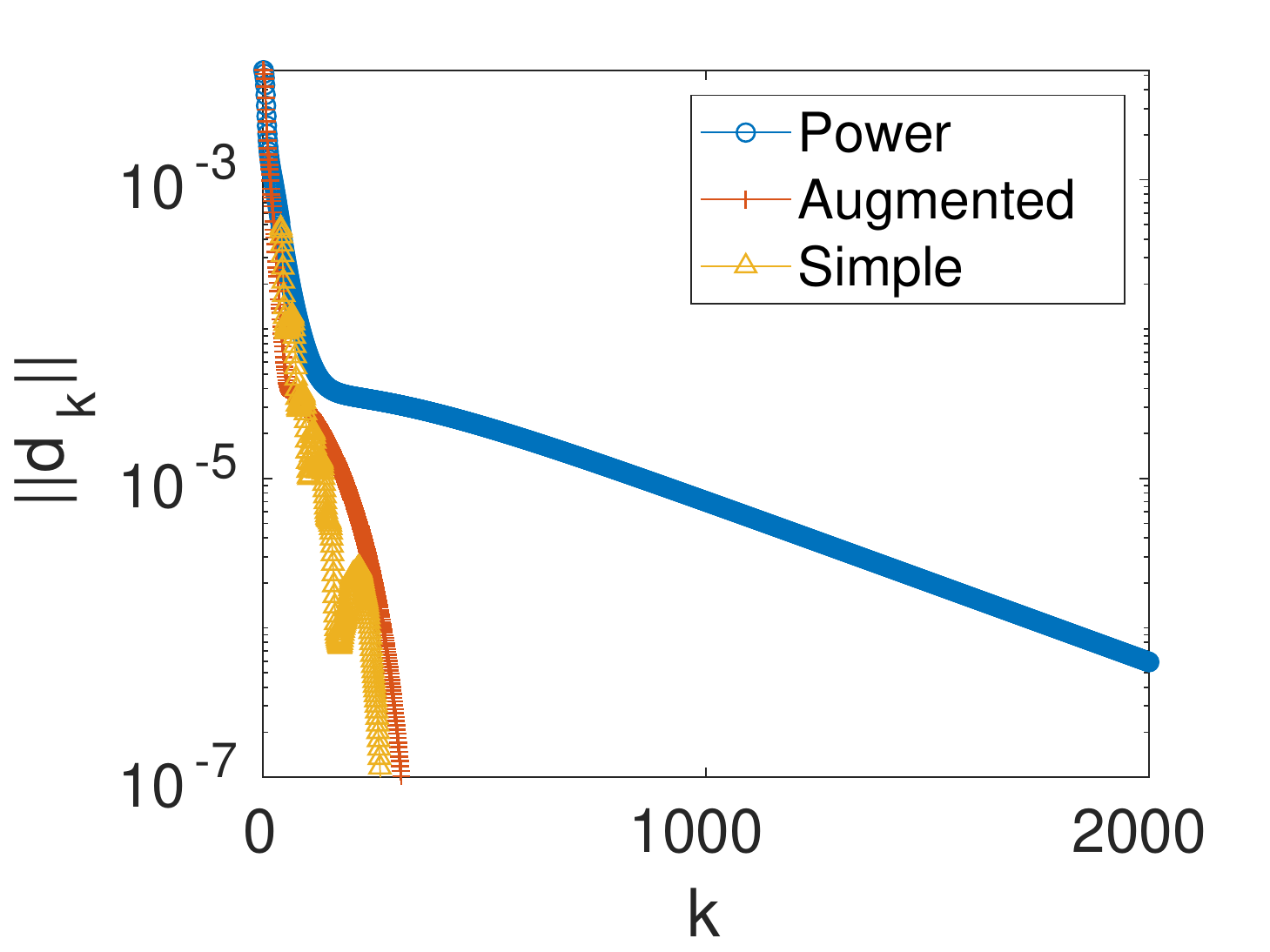}
\caption{Residual history for the Wilkinson matrix {\tt W21+}, with $n=21$ (left), 
a diagonal matrix with $n=1000$, leading eigenvalue 1 and remaining eigenvalues 
equally spaced between $0.75$ and $0.999$ (center), 
and a sparse matrix with $n=102158$ from a finite 
element problem describing temperature in a steel cylinder (right).}
\label{fig:ex1}
\end{figure}
This part of the example demonstrates the methods on benchmark 
problems of three different sizes.
\begin{itemize}
\item Matrix 1
is the Wilkinson matrix {\tt W21+}, which can be generated in Matlab by 
{\tt gallery('wilk',21)}. It is   
a tridiagonal matrix with pairs of nearly equal eigenvalues. The largest eigenvalues
are approximately $10.746$, and all but one of the eigenvalues are positive, 
with the negative eigenvalue approximately $-1.125$. 

\item Matrix 2 is a diagonal matrix of order $n=1000$, 
with leading eigenvalue 1, and remaining eigenvalues equally spaced between
$0.75$ and $0.999$, defined in Matlab by {\tt diag(v)}, with 
${\tt v = [1,linspace(0.75,0.999,1000)]}$.

\item Matrix 3 is a sparse matrix of order $n = 102158$, from a finite element problem
describing temperature a steel cylinder. 
It has $r \approx 0.9975$.
This matrix is available as 
{\tt thermomech\_TC}, from the SuiteSparse matrix collection \cite{DH11}.
\end{itemize}
Figure \ref{fig:ex1} shows the residual histories for each of these three matrices 
starting from an initial vector $u_0 = [1,1, \ldots, 1]$.

On each of these examples both the simple and augmented acceleration outperform the
power method; and, the augmented method demonstrates the exponential convergence
predicted from the theory.  The simple method shows more oscillatory behavior in the
second two cases, as additional eigencomponents play a dominant role, and are 
subsequently damped, as the iterations progress.

As a demonstration of robustness with respect to initial data, 
algorithms \ref{alg:power}-\ref{alg:simp-aug} are run on Matrix 1,2 and 3, 
starting with 100 different initial vectors $u_0$ determined by
the Matlab command {\tt rand(n,1)-0.5}. 
The augmented method alg. \ref{alg:simp-aug} is run with three different values
of the tuning parameter $\eta \!=\! \{20,40,80\}$.
As above, the simple method is started after 40 initial power iterations. 
Runs were terminated after a maximum
of 6000 iterations.  
The average number of iterations to residual convergence 
of {\tt tol} = $10^{-7}$ is reported in table \ref{tab:rand-init}.
The results show general agreement with those shown in figure \ref{fig:ex1}, 
demonstrating the methods are not overly sensitive to choice of initial iterate.
The last three columns of table \ref{tab:rand-init} show the computation with the three
different values of $\eta$ in the augmented method.  We see that there does appear to
be a best value for each problem, but the method is not overly sensitive to the choice.
Problems with many components competing for dominance, like Matrix 2, appear to benefit
from larger values, while problems featuring larger spectral gaps like Matrix 3
(see the performance of the power method in figure \ref{fig:ex1}) may show better efficiency
with smaller values.
\begin{table}
\begin{center}
\begin{tabular}{c|rrrrr}
Algorithm& \ref{alg:power} & \ref{alg:simple} & \ref{alg:simp-aug}, $\eta=20$ &
\ref{alg:simp-aug}, $\eta = 40$ & \ref{alg:simp-aug}, $\eta = 80$ \\
\hline
Matrix 1 & 107.6 & 58.8    & 58.6   & 42.9   & 42.1 \\
Matrix 2 & $>$6000 & 4295.1& 1457.4 & 1058.6 & 998.3\\
Matrix 3 & 2601.2 & 254.0  & 250.4  & 296.8  & 371.3
\end{tabular}
\end{center}
\caption{Average number of iterations to convergence of {\tt tol}= $10^{-7}$
over 100 initial iterates,
for the power method (alg. \ref{alg:power}), simple method (alg. \ref{alg:simp-aug}) and 
augmented method (alg. \ref{alg:simp-aug}).}
\label{tab:rand-init}
\end{table}

For the final example in this subsection we consider the comparative 
performance of the methods on a parameterized family of non-normal matrices.
The example is constructed so the non-normal part (nonsymmetric, in this case) 
becomes large compared to the normal (symmetric) part as the parameter $t$ increases.
While the assumptions behind the analysis in section\ref{sec:ideal} and \ref{sec:simple} 
include the orthogonality between the
dominant eigenvectors, the normality of system matrix $A$ was not assumed. 
Here, a parametrized matrix $A = A_t$ is given by
\begin{align}\label{eqn:At}
A_t = \begin{pmatrix}
\ddots & \ddots \\ & a_0 & t a_1 \\ & & \ddots & \ddots
\end{pmatrix}, 
~a_0(j) = j, ~ j = 1, \ldots, 100,
~a_1(j) = \left\{\begin{array}{cr}
1, & 1 \le j \le 50\\
0, & j > 50
\end{array} \right. .
\end{align}

Table \ref{tab:nnorm} shows the iterations to achieve a residual tolerance of $10^{-7}$,
from a starting vector of ones. The simple method is run with 40 initial power 
iterations, and the augmented method is run with $\eta = 40$.
\begin{table}
\begin{center}
\begin{tabular}{c|rrrrrrr}
$t$& 1 & 4 & 16 & 64 & 256 & 1024 & 4096\\
\hline
Alg. \ref{alg:power}    & 1604 & 1604 & 1604 & 1604 & 1604 & 1604 & 1604 \\ 
Alg. \ref{alg:simple}   & 580  & 580  & 580  & 399  & 544  & 650  & 829\\
Alg. \ref{alg:simp-aug} & 388  & 388  & 388  & 402  & 526  & 666  & 657 
\end{tabular}
\end{center}
\caption{Number of iterations to convergence of {\tt tol}= $10^{-7}$ for the 
matrix $A_t$ of \eqref{eqn:At},
for the power method (alg. \ref{alg:power}), simple method (alg. \ref{alg:simple}) 
and augmented method (alg. \ref{alg:simp-aug}).}
\label{tab:nnorm}
\end{table}
In constrast to the behavior of the inverse iteration for this problem
({\em cf.,} \cite{ipsen97}), the residual in this example is a good indicator of 
convergence to an analytical eigenpair for both eigenvalues and eigenvectors.
For each of the results shown in table table \ref{tab:nnorm}, 
the eigenvalue error is on the order
of $10^{-14}$, with corresponding eigenvector error on the order of $10^{-8}$.
Each iteration results in convergence to the dominant eigenpair with $\lambda = 100$.
\subsection{Example 2: Bad initial data}\label{subsec:ex2}
We illustrate section \ref{sec:augment} with some simple but extreme tests starting
with bad initial data. Figure \ref{fig:test0dk}, shows the performance of the three 
algorithms on $A=\diag([1,2, 0.01])$ with an initial iterate of 
$u_0=[0.01,0.01, 10^8]$, and 
the parameter $\eta=1$. The simple iteration is started 
without any additional power iterations. 
The initial iterate $u_0$ has its largest 
component in the wrong direction. Starting with iterate 3, the scale factors 
$\gamma_k$ are different for the two methods, and these in turn lead to different 
dynamics for the eigenvector approximations. But, we see the augmented and simple 
effectively forcing successive approximations to align with the dominant eigenvector 
direction $[0,1,0]$.
 \begin{figure}
  \centering
  \includegraphics[width=0.4\linewidth]{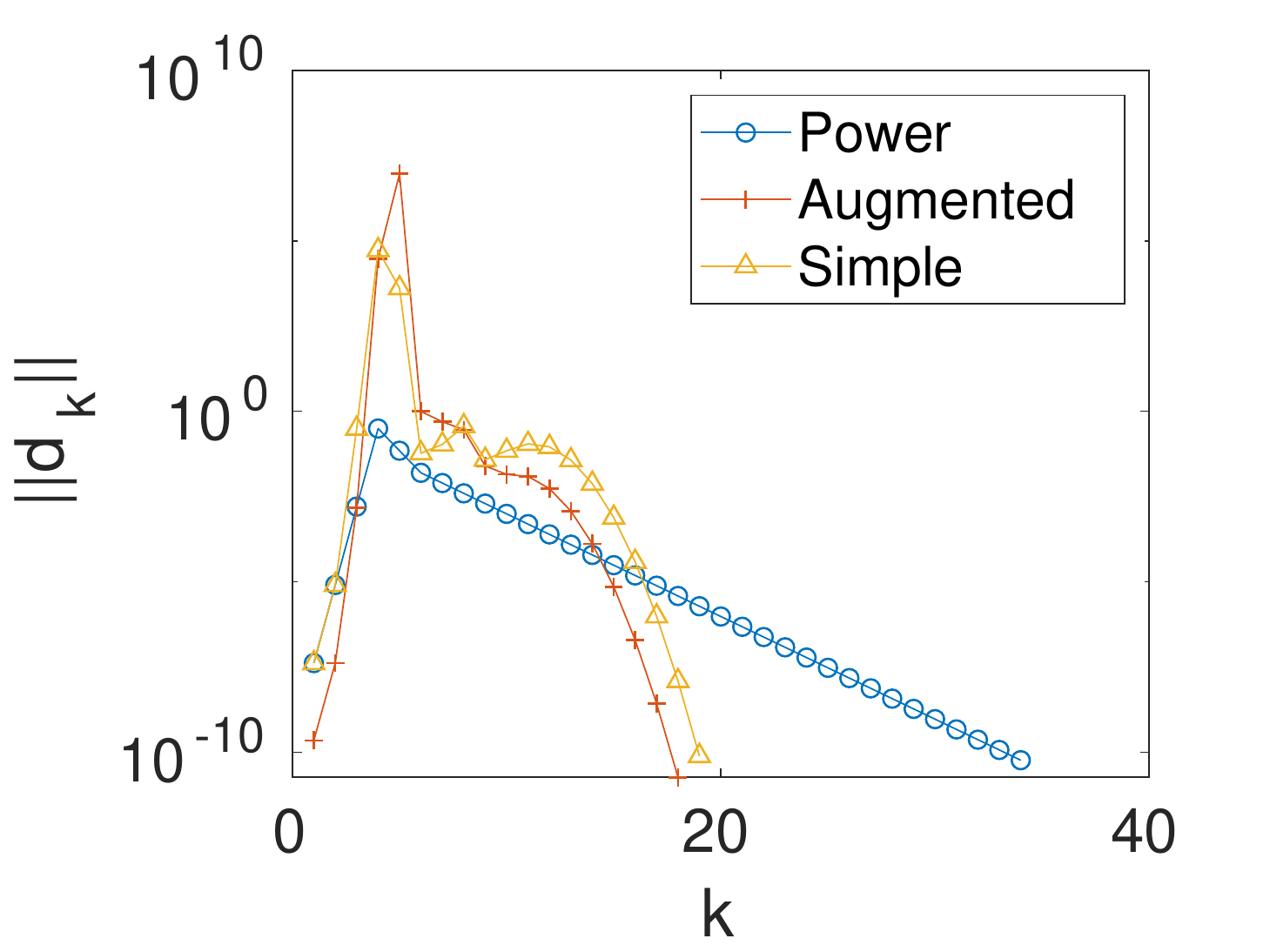}
  \includegraphics[width=0.4\linewidth]{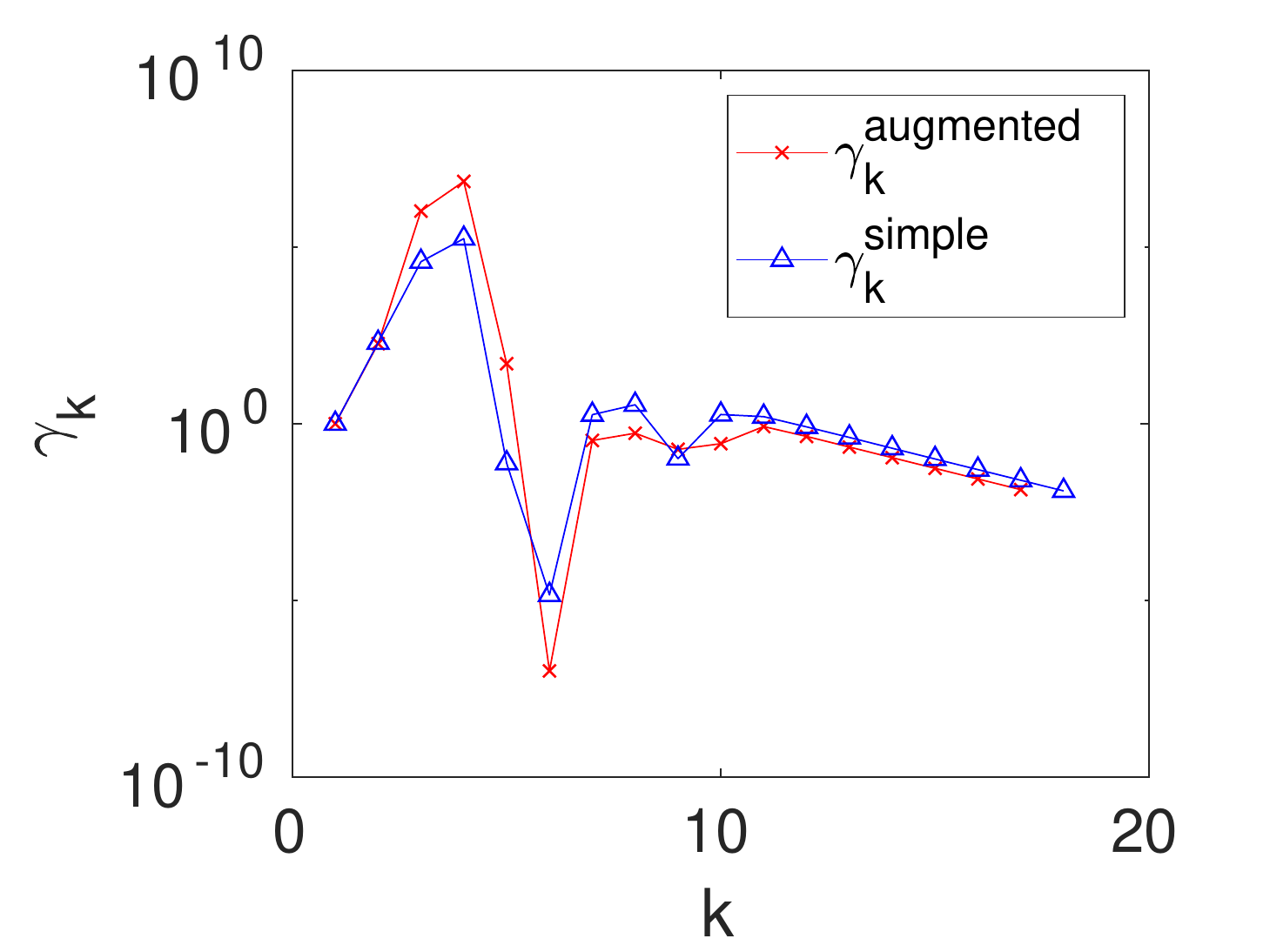}
  \caption{Residual histories (left) and extrapolation parameters (right) for 
  $A=\diag([1,2, 0.01])$ with an initial iterate of $u_0=[0.01,0.01, 10^8]$.}
  \label{fig:test0dk}
\end{figure}

The next example again features an extremely poor initial iterate.  Additionally,
the matrix $A=\diag([1.01, 1,0.1, 0.01]),$ features a small spectral gap. 
Here, $u_0=[0.01, 0.01,1, 10^9]$, so $c_M=10^9$, corresponds to the eigenvalue $0.01$, 
and the coefficient $c_1$ corresponding to the leading eigenvalue is $0.01$. 

The sequence of approximate (normalized) eigenvectors generated by each method are 
presented componentwise in figure \ref{fig:test1c}, to give a detailed view of how the
methods compare.
All methods quickly resolve the initially bad data, 
by damping out the fourth component (right plot) 
within the first few iterations, meanwhile
increasing the component in the dominant direction (left plot). 
The two accelerated methods are much more efficient (center plot) at damping out the 
second component, corresponding to the second eigenvalue, which is close to the first.

\begin{figure}
 \centering
 \includegraphics[width=0.3\linewidth]{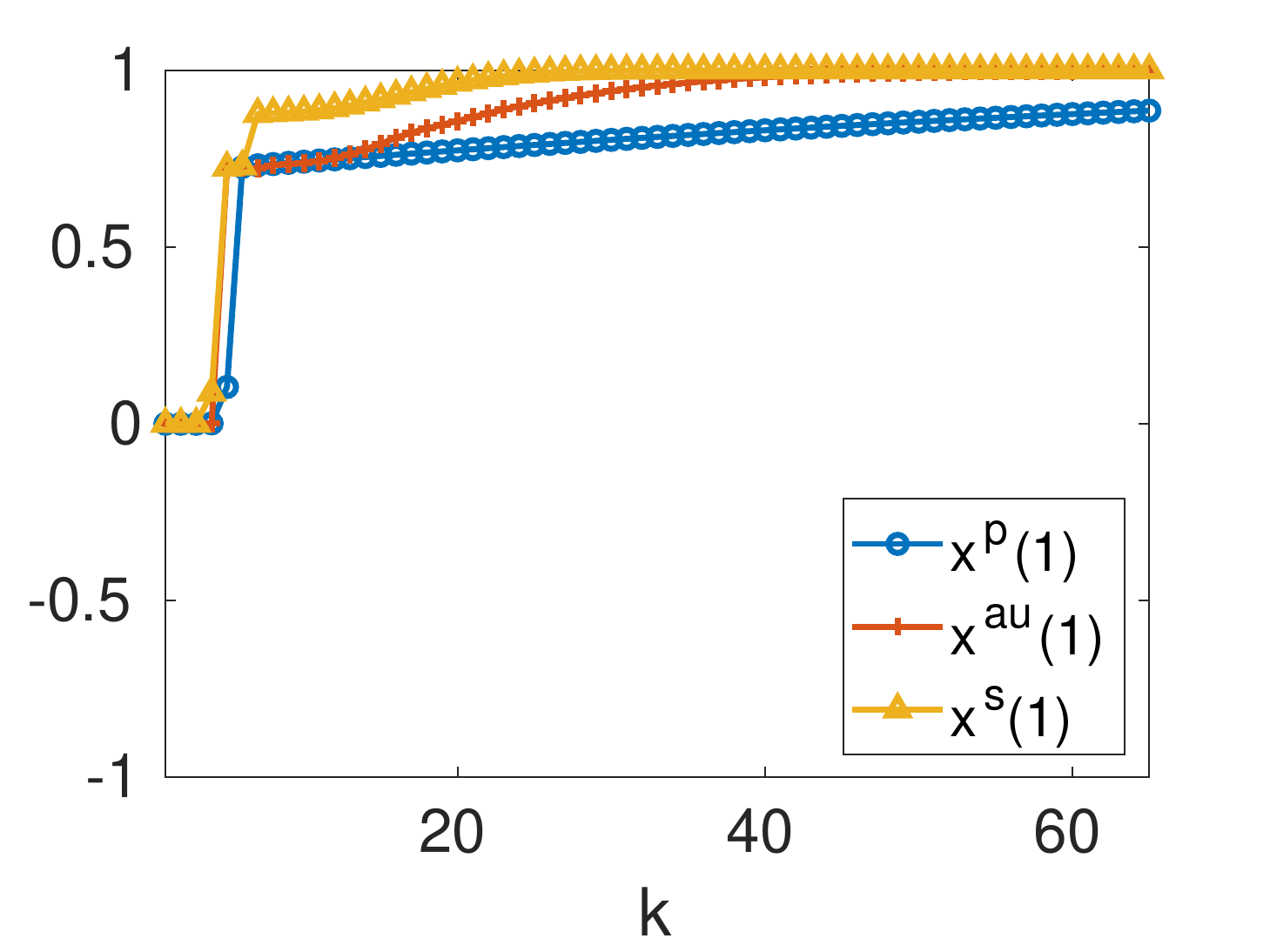}
 \includegraphics[width=0.3\linewidth]{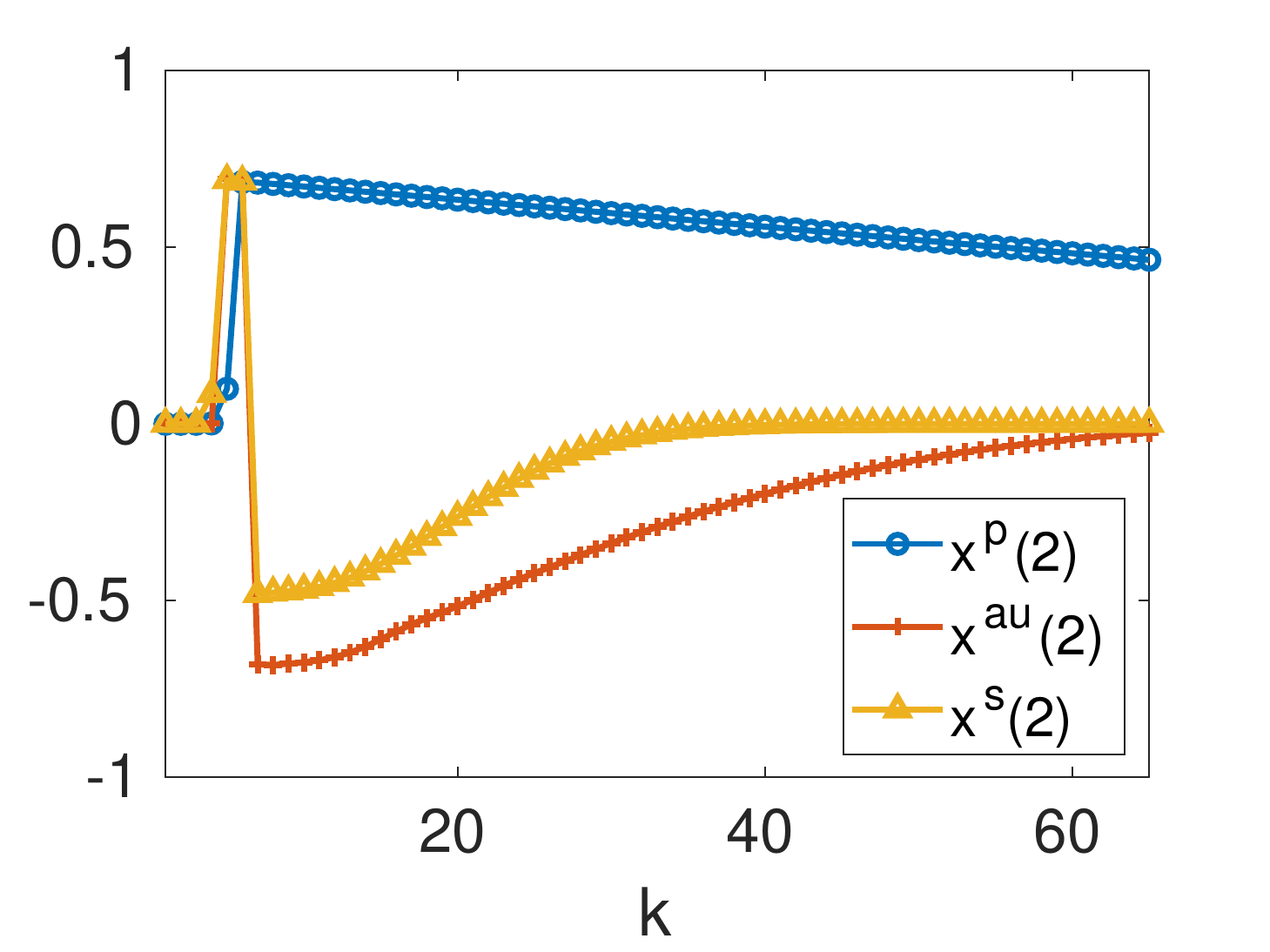}
 \includegraphics[width=0.3\linewidth]{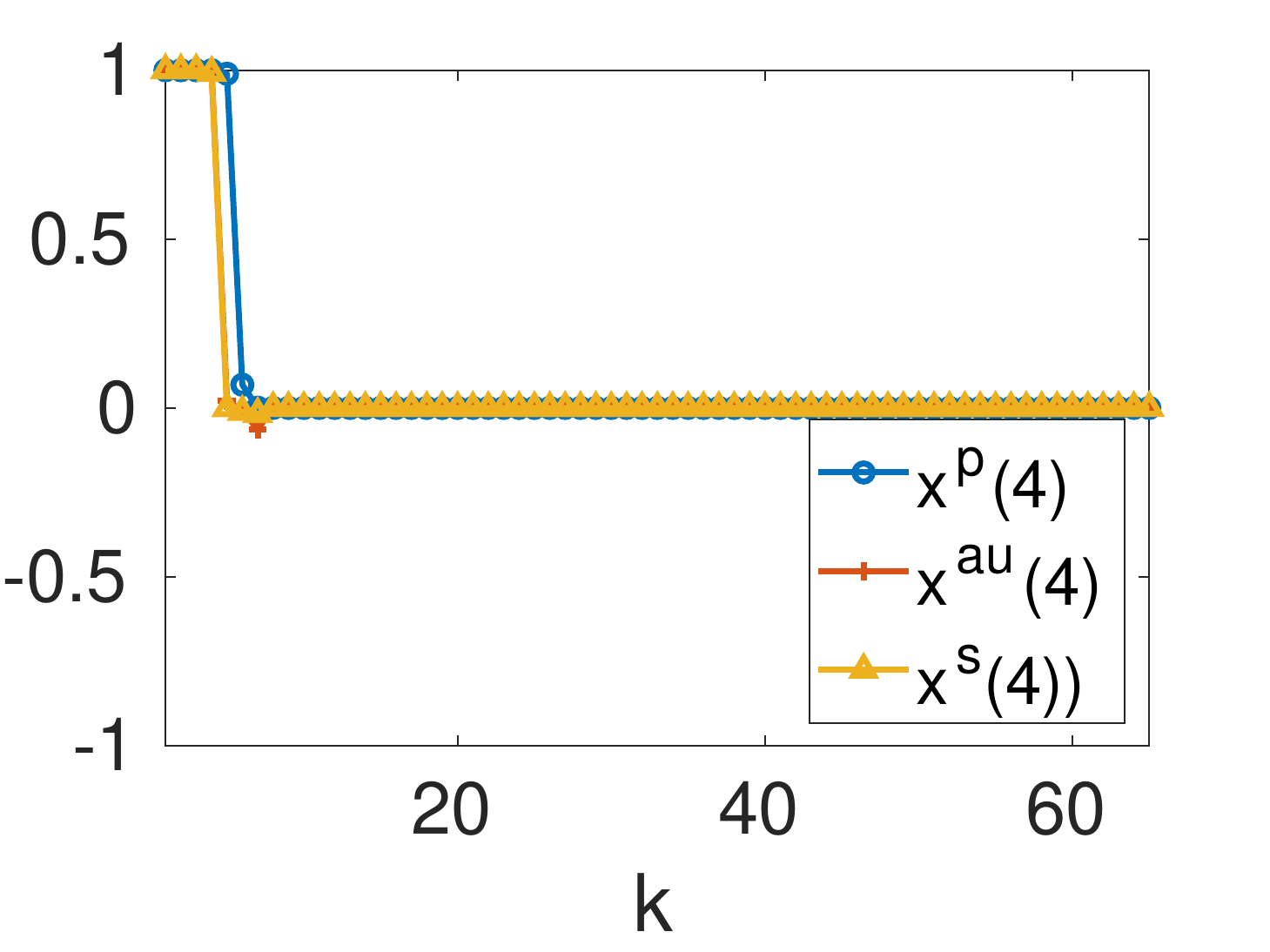}
  \caption{First (left), second (center) and fourth (right) components of the normalized 
  iterates for the power method $x^p$, augmented method $x^{au}$ with $\eta=10$, and
  the simple method $x^s$, for 
  $A=\diag([1.01, 1, 0.1, 0.01])$ with an initial iterate of $u_0=[0.01,0.01,1,10^9]$.
  }
\label{fig:test1c}
\end{figure}

\subsection{Comparison with Krylov subspace iterations}\label{subsec:ifpks}
In this section the runtimes for converging to dominant eigenpairs 
for algorithms \ref{alg:simp-aug} and \ref{alg:simple} are compared with the inverse 
free preconditioned
Krylov subspace method (EIGIFP), \cite[Algorithm 4]{GoYe02}, and the block inverse-free
preconditioned Krylov subspace method (BLEIGIFP) \cite[Algorithm 5]{QuYe10}. 
The second of these
two is designed to improve performance in the case of repeated or clustered eigenvalues.
Both algorithms are provided by their authors as stand-alone routines made available 
as m-files, from \url{https://www.ms.uky.edu/~qye/}.

The first matrix tested is Matrix 2 from subsection\ref{subsec:ex1}, with $n=1000$.  The
remaining tests are from the SuiteSparse matrix collection \cite{DH11}:
{\tt Si5H12} with $n= 19896$; 
{\tt c-65} with $n = 48066$;
{\tt Andrews} with $n = 60000$; and
{\tt thermomech\_TC} with $n = 102158$, which was also used in Example 1.

Each iteration was started with uniformly distributed initial data between $-1/2$ and 
$1/2$. Algorithm \ref{alg:simp-aug} was run with parameter $\eta = 45$, and 
alg. \ref{alg:simple} was started after 40 initial power iterations on each run.
EIGIFP and BLEIGIFP were run with their default parameter sets with output suppressed, 
in all cases to a tolerance of $10^{-7}$. 
Timing was performed using Matlab's {\tt tic} and {\tt toc}
commands.  The tests were done using Matlab R2018a on an 8 core intel Xeon W-2145 
CPU @ 3.70 GHz, with 64 GB memory. The results are summarized in table \ref{tab:runtimes}.
\begin{table}
\begin{center}
\begin{tabular}{c|rrrrr}
Algorithm& Alg. \ref{alg:simple} & Alg.\ref{alg:simp-aug} & EIGIFP & BLEIGIFP \\
\hline
Matrix 2            & 3.89e-02 &  1.03e-02 & 9.00e-03 & 2.70e-02 \\
{\tt Si5H12}        & 1.73e+00 &  7.22e-01 & 2.25e+00 & 6.81e-01 \\
{\tt c-65}          & 2.91e-02 &  1.45e-02 & 4.12e+00 & 8.55e-02 \\
{\tt Andrews}       & 8.30e+00 &  2.39e+00 & 1.95e+01 & 1.84e+00 \\
{\tt thermomech\_TC}& 6.60e-01 &  8.47e-01 & 3.63e+01 & 8.31e-01
\end{tabular}
\end{center}
\caption{Average time (sec) to convergence of {\tt tol}= $10^{-7}$
over 100 initial iterates,
for the simple method, augmented method, EIGIFP and BLEIGIFP.}
\label{tab:runtimes}
\end{table}

EIGIFP performed nominally better than the others for Matrix 2.
Algorithm \ref{alg:simp-aug} and BLEIGIFP were the fastest in the middle 3 cases, with 
BLEIGIFP running in respsectively $94\%$ and $77\%$ of the time of 
alg. \ref{alg:simp-aug} for {\tt Si5H12} and {Andrews}; and alg. \ref{alg:simp-aug} 
ran in $17\%$ the time of BLEIGIFP for {\tt c-65}. 
Algorithm \ref{alg:simple} ran somewhat faster than either for {\tt thermomech\_TC}.
From these results, alg. \ref{alg:simp-aug} and BLEIGIFP are roughly comparable for 
finding a single dominant eigenpair in these tests, with alg. \ref{alg:simple} and 
EIGEIF occasionally faster but potentially substantially slower on the problems
tested.

\subsection{Laplace-Neumann eigenvalues on the unit square}\label{subsec:LapNeu}

In this example, we seek a specific Neumann eigenvalue on the unit square $\Omega$, 
{\em i.e.}, eigenpairs of 
\begin{equation}
\label{eqn:neumann_ep}
-\Delta u_m =\mu_m u_m, ~  x \in \Omega, \qquad
\frac{\partial u_m}{\partial n}  = 0, ~  x \in \Gamma.
\end{equation}
The eigenvalues $\mu_m$ are of the form 
$(k^2+\ell^2)\pi^2, k,\ell=0,1,2,....$ The first several eigenvalues, 
scaled by $\pi^2$ for easier reading, are 
$0,1,1,2,2,4,4,5,5,8,9,9,\ldots$.

We discretize the domain $\Omega$ using a $P_2$ Lagrange finite element method within the {\tt FreeFem++} library \cite{hecht12}, using 20 nodes per side on the boundary. 
This leads to the discrete linear system $K{x} = \mu_h M {x}$, 
where $K$ and $M$ are the usual finite element stiffness and mass matrices, both of which are sparse and symmetric and of size $n = 1681$. 
With the relatively coarse mesh, we only expect the lower eigenmodes to be captured to 
high accuracy. We also note some eigenvalues occur with multiplicity 2.
For this problem, the stiffness matrix is singular, and the eigenvector corresponding to 
$\lambda = 0$ is a constant function on $\Omega$, which explains the initially very 
small residual when the iteration is started with an initial iterate of ones. 

We first implement our accelerated methods using
$A=M\backslash K$, with $\eta=10$, {\tt tol}=$10^{-10}$, a starting iterate of 
$u={\tt ones}(N,1)$, and a maximum of 400 iterations. 
We do not expect the largest discrete eigenvalue to be very close to the continuous one 
of a similar size $(5426\pi^2)$ since the discretization is coarse.
Both the augmented and simple methods converge to 
$5439.464585998008\pi^2$. 
The power method converges to a similar value, but with a much larger residual, 
suggesting greater inaccuracy in the eigenvector.
The results are shown in figure \ref{fig:neumannfemfull4dk} on the left.

\begin{figure}
  \centering
  \includegraphics[width=0.4\linewidth]{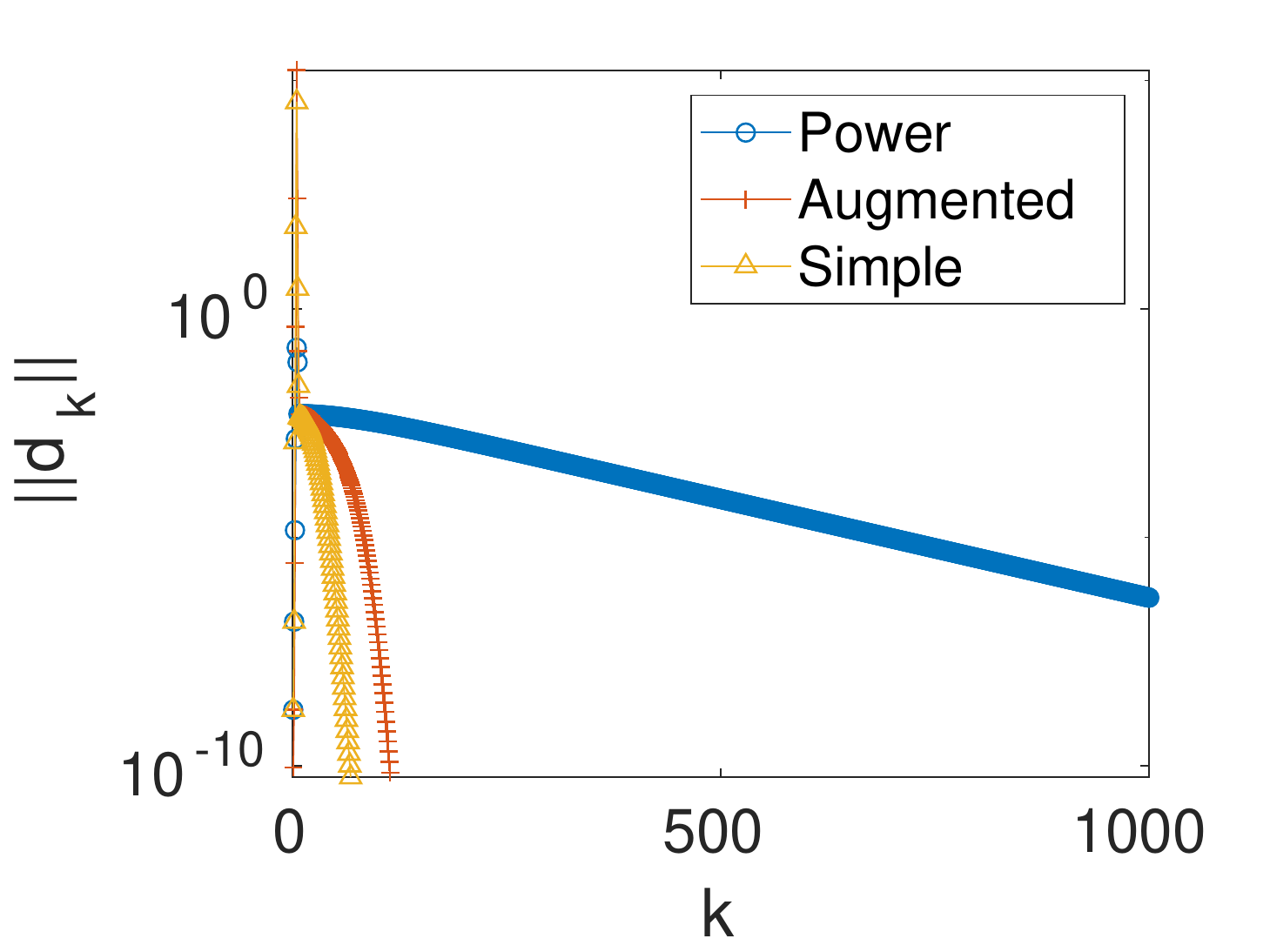}
  \includegraphics[width=0.4\linewidth]{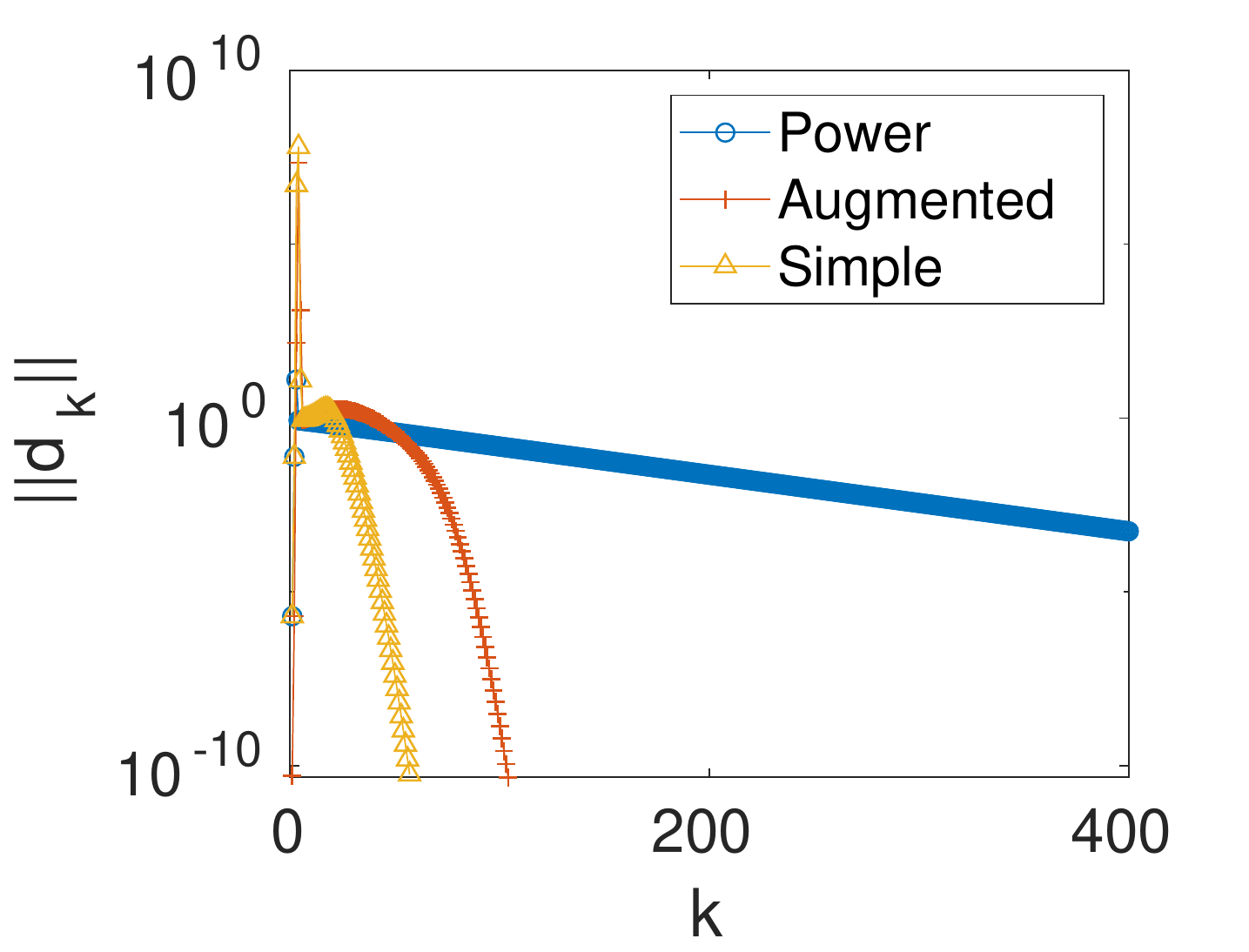}
  \caption{Residual histories for the 
  largest eigenvalue (left), and eigenvalue nearest to $20\pi^2$ (right),
  for problem \eqref{eqn:neumann_ep}.}
  \label{fig:neumannfemfull4dk}
\end{figure}

We next implement our accelerated method to locate the eigenvalue closest to the pair at 
$ 20\pi^2$. Using Matlab's {\tt eig} with an  shift of 
$20.01\pi^2$ and the default tolerance of $10^{-14}$, we obtain that the nearest 
discrete eigenvalue as $\mu_h=20.009477014838016\pi^2$. 
We run our method on 
$A=(K-\text{shift}\times M)\backslash M$, with 
$\eta=10$, {\tt tol}=$10^{-10}$, and a maximum of 400 iterations. 
The progression of the methods are shown in 
figure \ref{fig:neumannfemfull4dk} (right).
Both simple and augmented methods quickly recover from the initial data along 
the component of the zero eigenvalue and demonstrate exponential convergence.
The simple method converges to the same eigenvalue as the augmented method. The power iteration has a larger error, with the located eigenvalue
agreeing with the augmented and simple methods to $10^{-11}\pi^2$,
but with a much larger residual, suggesting an inaccurate approximation to the
eigenvector.

We recall here the algorithms presented in this paper assume the first few dominant eigenvalues are positive. When applying a shift, this condition 
is easily violated, particularly early in the spectrum. So, for instance, if we seek the 
eigenvalues of the Neumann problem closest to $4\pi^2$ by using the true eigenvalue as a 
shift, the methods do not work. 

Our final FEM example is for the Neumann eigenproblem when $\Omega$ consists of two 
squares of sides length 2, connected by a thin channel of thickness 0.01 and length 2, 
with 30 grid points per edge of the polygon, shown in figure \ref{fig:channel-mesh}. 
In this instance, the eigenfunctions are not guaranteed to be regular, and we expect some eigenvalues to have high multiplicity. 

We set the maximum number of iterations to 1300. Using $P_2$ Lagrange elements, 
we are lead to a $5461\times 5461$ sparse system 
${ K {x}= \mu_h M {x}}$. 
Matlab's {\tt eigs} yields the largest eigenvalue as $\approx 20081.60885\pi^2$. The simple and augmented methods yield the same eigenvalue, as does the power method; in this instance, the residual is least when using the augmented strategy.  The residual histories for
the three methods are shown in figure \ref{fig:channel},
along with a comparison between the residual and projection used
to compute the extrapolation parameter $\gamma$, in the augmented method.
 
\begin{figure}
  \centering
  \includegraphics[trim = 100pt 130pt 100pt 130pt,clip = true, width=0.7\textwidth]
   {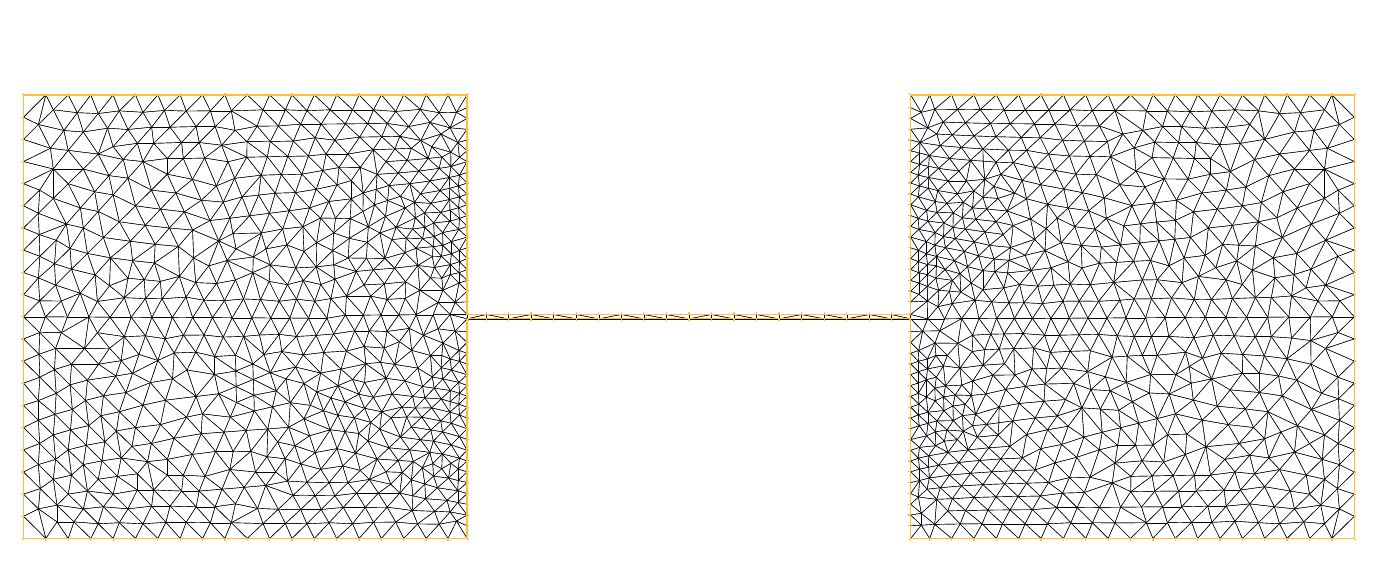}
  \caption{Domain $\Omega$ and its mesh for the channel problem.}
  \label{fig:channel-mesh}
\end{figure}
\begin{figure}		
  \centering
  \includegraphics[height=0.3\linewidth]{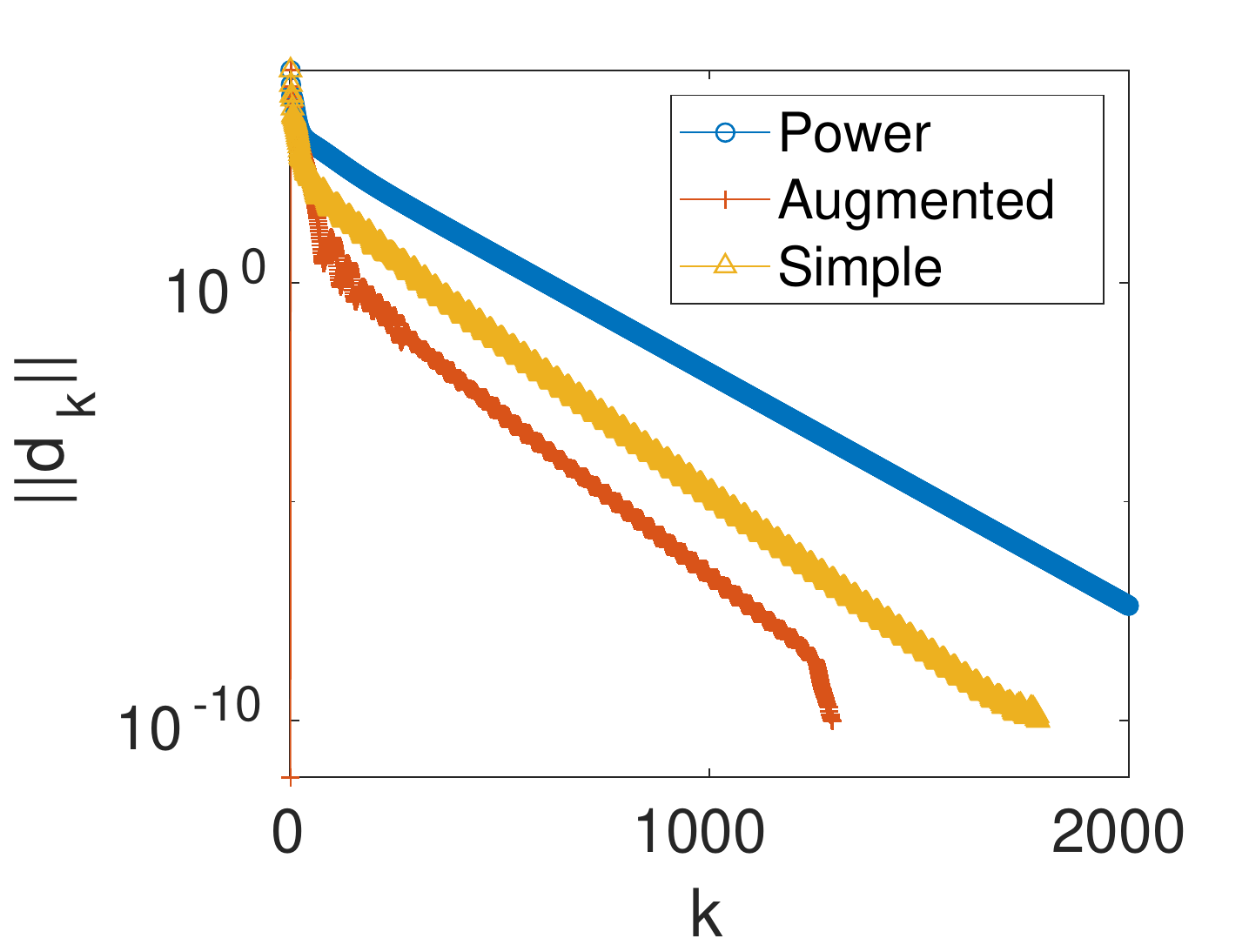}
  \includegraphics[height=0.3\linewidth]{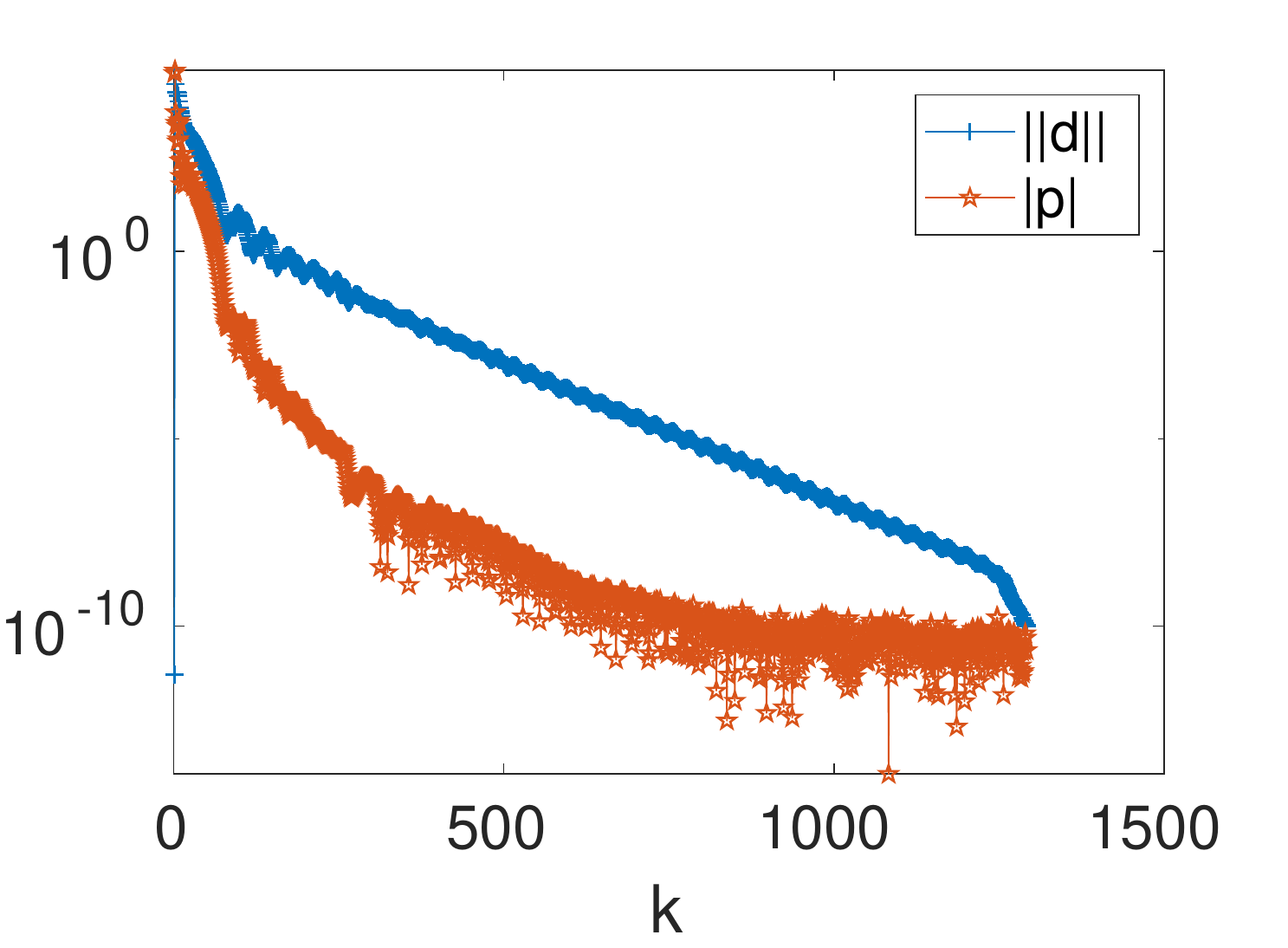}
  \caption{Residual histories for highest computed eigenvalue in 
   the channel example (left), and the residual $\nr{d_k}$ compared to the 
   projection $p_k$ used in the augmented method.}
  \label{fig:channel}
\end{figure}

\subsection{Steklov eigenvalues}
\begin{figure}
	\centering
	\includegraphics[height = 0.3\textwidth]{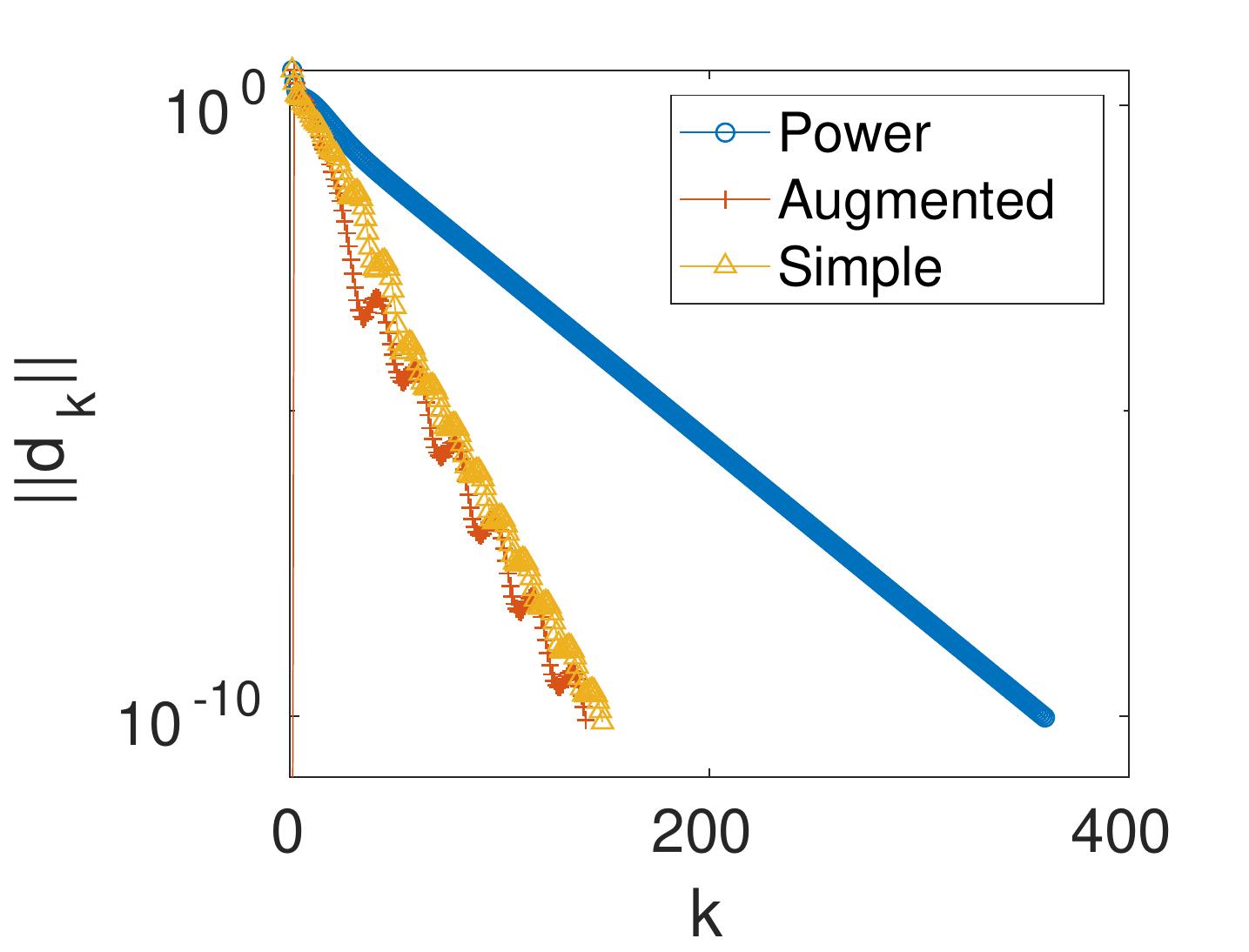}
	\includegraphics[height=0.3\textwidth]{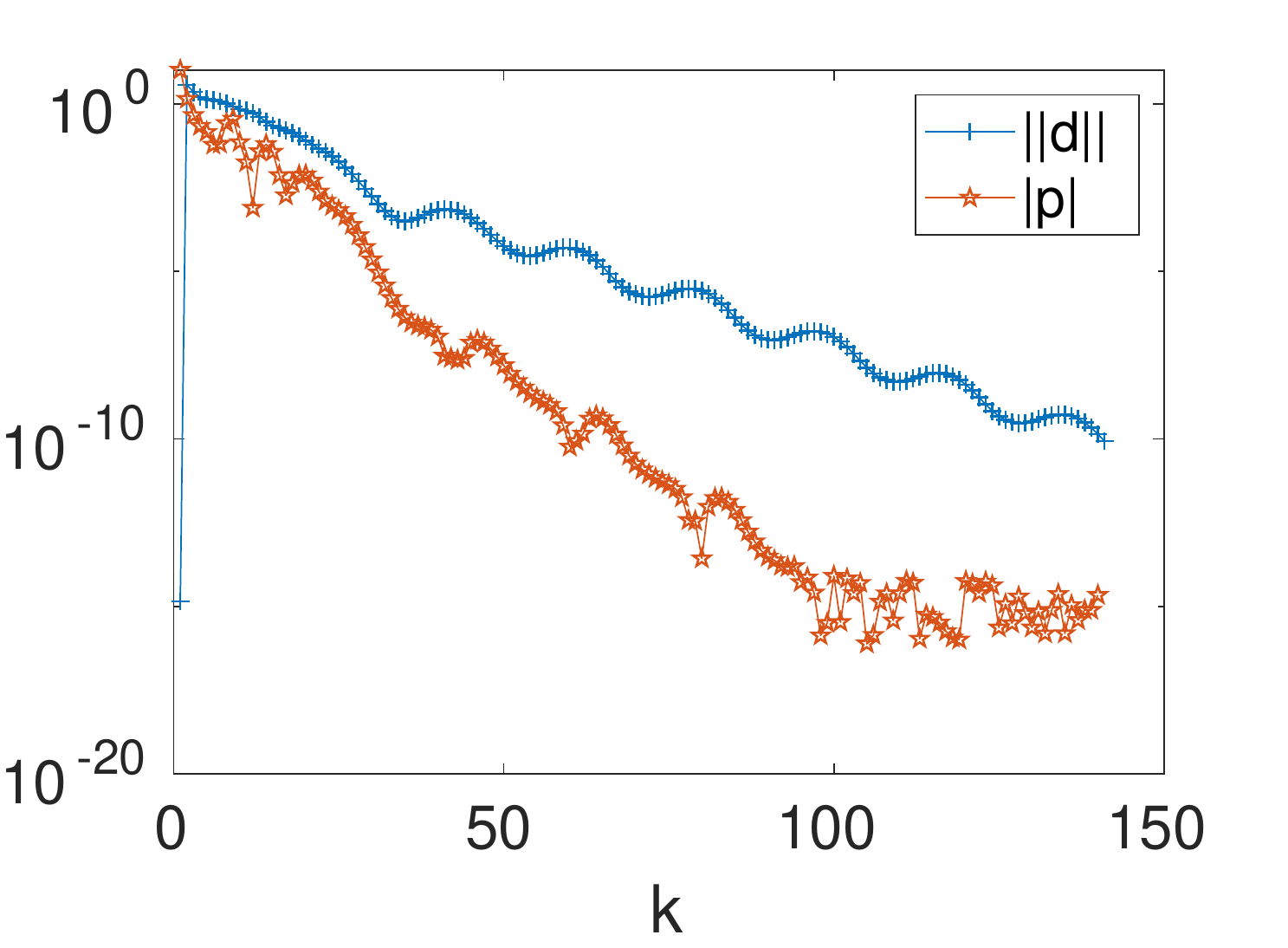}
	\caption{Residual histories (left) and the residual $\nr{d_k}$
 compared to the projection $p_k$ used in the augmented method, for the Steklov
 problem.}
  \label{fig:steklovdk}
\end{figure}
In this last example, we examine the behavior of our acceleration scheme on an 
eigenvalue problem involving the eigenvalues of a dense, and possibly ill-conditioned 
matrix.

Let  $\Omega\subset \mathbb{R}^2$ be  a bounded domain with a Lipschitz piecewise-smooth
boundary $\Gamma$. The general Steklov  eigenvalue problem can be stated as: 
find $u_m\in H^1(\Omega), \lambda_m \in \mathbb{R} $ so that
\begin{equation}
\label{steklov_ep}
\Delta u_m = 0 ~  x, \in \Omega, \qquad
\frac{\partial u_m}{\partial n}  = \lambda_m  u_m, ~  x \in \Gamma,
\end{equation}
for eigenvalues $\lambda_m$ and eigenfunction $u_m$. The system \eqref{steklov_ep} is called the {\it Steklov} problem, and $\lambda_m$ and $u_m$ are called Steklov eigenvalues and eigenfunctions. We note this spectrum coincides with that of the Dirichlet-to-Neumann operator
$\Lambda\colon H^{\frac{1}{2}}(\Gamma)\to H^{-\frac{1}{2}}(\Gamma)$,
given by $\Lambda u=\partial_{n}(\mathcal{H}u)$, where $\mathcal{H} u$
denotes the unique harmonic extension of $u\in
H^{\frac{1}{2}}(\Gamma)$ to $\Omega$. 

A particularly elegant and accurate (in the sense of approximation) strategy for computing $\lambda_m$ is {\it via} boundary integral strategies. We shall employ a {\it modified} single layer strategy.
A modified formulation is based on the ansatz
\begin{equation}
\label{steklov_single_layer_modified}
u(x) = \int_{\Gamma} \Phi (x-y)(\varphi(y) - \overline{\varphi} ds(y) + \overline{\varphi}  \quad x\in \Omega,
\end{equation}
based on the average of the density
$\overline{\varphi} = |\Gamma|^{-1}\int_{\Gamma} \varphi(y)  ds(y)$,
as suggested in~\cite[Equation 7.58]{kress1999}. Here $\Phi$ denotes the fundamental solution for the Laplacian. We  introduce the operators~ $\mathcal{S}:H^{-1/2}(\Gamma)\rightarrow H^{1/2}(\Gamma)$ and $\mathcal{T}:H^{-1/2}(\Gamma)\rightarrow H^{-1/2}(\Gamma)$ as 
\begin{align*} 
\mathcal{S}[\phi]\coloneqq \int_{\Gamma}  \Phi(x-y)\phi(y) ds(y), 
\quad
\mathcal{T}[\phi]\coloneqq \int_{\Gamma}  \frac{\partial\Phi(x-y)}{\partial n(x)}\phi(y)ds(y), \quad x\in \Gamma.
\end{align*}
Taking into
account well known expressions (see e.g.~\cite{kress1999}) for the
jump of the single layer potential and its normal derivative across
$\Gamma$, the eigenvalue problem \eqref{steklov_single_layer_modified} is
reduced to a system of integral equations 
\begin{equation}
\label{steklov_integral_equation_system}
(\mathcal{T}+\frac{1}{2}\mathcal{I})\left[ \varphi - \overline{\varphi}\right]=
\lambda\left(\mathcal{S}[\varphi-\overline{\varphi}]+ \overline{\varphi} \right),x\in\Gamma,
\end{equation}
for the eigenvalue $\lambda$ and density $\phi$. This system is discretized using a Fourier spectral strategy. The resultant discrete  generalized eigenvalue problem is to find $\lambda \in\mathbb{R}, c \in \mathbb{R}^N$ such that
${A c} = \lambda {Bc}$,
where ${B}$ is invertible (thanks to the modification of the single layer). Both matrices are dense.

We take $\Omega$ to be the unit disk, and use $N=32$ collocation points on $\Gamma$. The true eigenvalues are known to consist of the countable set $0,1,1,2,2,3,3....$; with the discrete strategy used, we expect the largest eigenvalue we approximate to be close to $16$.
All three tested methods locate the eigenvalue correct to 
14 digits. (We have used a tolerance of $10^{-10}$). 
The rapid decay of the residual is evident in the accelerated methods 
(see figure \ref{fig:steklovdk}, left).  
Once again, since $\lambda=0$ is an eigenvalue (corresponding to the constant 
eigendensity case), and our starting iterate of a vector of ones leads to an initial 
near-zero residual. 
The somewhat oscillatory convergence behavior is due to interference from 
eigencomponents other than the first two most dominant ones.  
This suggests that introducing an extrapolation
technique with greater depth that is designed to reduce multiple components at each 
step (instead of just the second-most dominant) may lead to smoother convergence.
Figure \ref{fig:steklovdk} on the right compares the residual to the projection used
in the augmented method.  Similarly to the previous example (see figure \ref{fig:channel},
right), here the projection $p_k$ of the update along the approximate eigenvector, 
converges at a better rate than the full residual; in contrast, however, it 
achieves a tolerance of $10^{-15}$, when the residual is reduced to $10^{-10}$.

\section{Conclusion}\label{sec:conc}
In this paper we introduced and analyzed alg. \ref{alg:simple}, a simple method
to accelerate the power iteration.  We proved the method features exponential
convergence if the initial iterate has components of only two eigenvectors.  
We further introduced alg. \ref{alg:simp-aug}, which
modifies the simple method to help stabilize the early stages of the iteration.
Both methods are one-step extrapolation techniques which form the accelerated
iterates from a linear combination of the two most recent update steps.
The extrapolation parameter is computed by a ratio of residuals, and
requires minimal additional computation beyond a residual norm and Rayleigh 
quotient at each iteration.  The methods are shown numerically to be robust with
respect to initial iterate, and to substantially improve performance in the case
where the spectral gap is small.


\bibliographystyle{siamplain}
\bibliography{eigrefs}
\vfill
\end{document}